\documentclass[12pt,a4paper]{amsart}

\usepackage{mystyle,mymacros}
\usepackage[all]{xy} 		
\usepackage{tabularx}
\usepackage{enumerate}
\usepackage{times}
\usepackage[ruled,vlined]{algorithm2e}
\usepackage{float}
\usepackage{tikz,graphicx}
\graphicspath{ {./images/} }

\makeatletter
\@namedef{subjclassname@2010}{%
 \textup{2010} Mathematics Subject Classification}
\makeatother

\title{$p$-adic integration on bad reduction hyperelliptic curves}

\author{Eric Katz}
\address{Eric Katz, Department of Mathematics, The Ohio State University, 231 W. 18th Ave., Columbus, OH 43210, USA}
\email{katz.60@osu.edu}

\author{Enis Kaya}
\address{Enis Kaya, Bernoulli Institute, University of Groningen, Nijenborgh 9, 9747 AG
Groningen, The Netherlands}
\email{e.kaya@rug.nl}

\newcommand{\Ab}{{\operatorname{Ab}}}
\newcommand{\lAb}{{\operatorname{Ab\ }}}
\newcommand{\BC}{{\operatorname{BC}}}
\newcommand{\lBC}{{\operatorname{BC\ }}}
\newcommand{\dR}{{\operatorname{dR}}}
\def\presuper#1#2%
  {\mathop{}%
   \mathopen{\vphantom{#2}}^{#1}%
   \kern-12\scriptspace%
   #2}
\newcommand{\defi}[1]{\textbf{#1}} 	
\newcommand\Abint{\presuper\Ab\int}
\newcommand\lAbint{\presuper\lAb\int}
\newcommand\BCint{\presuper\BC\int}
\newcommand\lBCint{\presuper\lBC\int}
\newcommand\tint{\presuper{t}\int}
\DeclareMathOperator{\PP}{PP}
\DeclareMathOperator{\Log}{Log}
\DeclareMathOperator{\Symb}{Symb}
\DeclareMathOperator{\iinv}{inv}
\mathtext{Diff}

\begin{document}

\begin{abstract} In this paper, we introduce an algorithm for computing $p$-adic integrals on bad reduction hyperelliptic curves. For bad reduction curves, there are two notions of $p$-adic integration: Berkovich--Coleman integrals which can be performed locally; and abelian integrals with desirable number-theoretic properties. By covering a bad reduction hyperelliptic curve by basic wide open sets, we reduce the computation of Berkovich--Coleman integrals to the known algorithms on good reduction hyperelliptic curves. These are due to Balakrishnan, Bradshaw, and Kedlaya, and to Balakrishnan and Besser for regular and meromorphic $1$-forms, respectively. We then employ tropical geometric techniques due to the first-named author with Rabinoff and Zureick-Brown to convert the Berkovich--Coleman integrals into abelian integrals. We provide examples of our algorithm, verifying that certain abelian integrals between torsion points vanish.
\end{abstract}

\maketitle

\setcounter{tocdepth}{1}
\tableofcontents

\section{Introduction}

The theory of $p$-adic line integrals on analytic curves was introduced by Coleman \cite{coleman85:effective_chabauty,
coleman85:torsion} to solve problems in number theory. In certain circumstances, it produces locally analytic functions that vanish on rational and torsion points on an algebraic curve $X$. These functions were originally defined by composing an Abel-Jacobi map $\iota$ on the analytification of $X$ with the $p$-adic logarithm on the Jacobian $J$ of $X$:
\[\xymatrix{X(\C_p)\ar[r]^{\iota}&J(\C_p)\ar[r]^{\Log}&\Lie(J).}\]
Because the $\Lie(J)$ is torsion-free, the torsion points of $J(\C_p)$ are necessarily taken to $0$ on $\Lie(J)$. In the Chabauty--Coleman method (see \cite{mccallum2012method}), under  conditions on the rank of $J(\Q)$, the image of $J(\Q)$ lies in a proper linear subspace of $\Lie(J)$. In either case, one hopes to find a linear function on $\Lie(J)$ whose pullback to the analytification $X^{\an}$ vanishes on the points of interest. It was Coleman's key insight that for good reduction curves, the function can be computed locally on $X^{\an}$ as an integral $f_\omega=\int \omega$ for $\omega\in \Omega^1(X)=\Lie(J)^\vee$. To define such an integral, one integrates $\omega$ on each residue disc and uses the Dwork principle, {\em analytic continuation by Frobenius}, to match the integrals between residue discs. Specifically, on an open affinoid subset $U\subset X^{\an}$, one can find a lift of Frobenius $\phi\colon U\to U$. Then, $f_\omega$ is determined up to a global constant by its being a local primitive of $\omega$, i.e., $df_\omega=\omega$ and by its obeying a change-of-variables formula with respect to $\phi$:
\[\int_S^R \phi^*\omega=\int_{\phi(S)}^{\phi(R)} \omega.\]
The function $f_\omega$ is independent of the choice of $\phi$. There are practical algorithms to compute $f_\omega$ when $p$ is odd, and they can be summarized as follows. 

\begin{itemize}
    \item Algorithms on odd degree hyperelliptic curves were developed in Balakrishnan--Bradshaw--Kedlaya \cite{BBKExplicit}. By drawing on Kedlaya's algorithm for computing the zeta function of hyperelliptic curves \cite{kedlaya2001counting}, one is able to choose an explicit lift of Frobenius and write down its action on Monsky--Washnitzer cohomology, a form of de Rham cohomology on affinoid spaces. One considers $H_{\dR}^1(U)^{-}$, the odd subspace of $H_{\dR}^1(U)$, that is, the $(-1)$-eigenspace of the hyperelliptic involution. Given a basis $\omega_1,\dots,\omega_k$ of $H^1_{\dR}(U)^{-}$, one writes $\phi^*\omega_i=df_i + \sum_{j} M_{ij}\omega_j$ for constants $M_{ij}$ and meromorphic functions $f_i$. From knowledge of $M_{ij}$, $f_i$, and the action of $\phi$ on points $R$ and $S$, one is able to solve for $\int_S^R \omega_i$. We also note that Best \cite{best2019explicit} has improved the complexity of the integration algorithms introduced in \cite{BBKExplicit}. 
    \item The paper \cite{BBKExplicit} came with certain restrictions. One is only able to integrate meromorphic $1$-forms whose poles are in residue discs around Weierstrass points. However, for some applications (e.g., computation of $p$-adic heights on curves as in \cite{coleman1989}), it is necessary to perform more general integrals. Such an algorithm was provided in Balakrishnan--Besser \cite{BBComputing} based on the theory of local symbols \cite{Besser2000}.
    \item Another restriction in \cite{BBKExplicit} is that the authors only deal with odd degree models; this restriction is inherited from \cite{kedlaya2001counting} where it is used out of convenience. Building on \cite{harrison2012extension}, in which Harrison adapted Kedlaya's algorithm to even degree case, Balakrishnan \cite{balakrishnan2015coleman} extended the integration algorithms in \cite{BBKExplicit} to even degree models of hyperelliptic curves.
    \item All the methods above only deal with hyperelliptic curves. Algorithms to carry out integration on more general curves were developed in Balakrishnan--Tuitman \cite{balakrishnan2020explicit} based on the work of Tuitman \cite{Tuitman:counting1, Tuitman:counting2} that generalizes Kedlaya's algorithm to this setting. We note that these algorithms work only for meromorphic $1$-forms that are holomorphic away from the ramification locus. 
\end{itemize}

The theory of $p$-adic integration is also useful in the bad reduction case as was demonstrated in the work of Stoll \cite{stoll2018uniform} and  of the first-named author with Rabinoff and Zureick-Brown \cite{krzb15:uniform_bounds}. Here, to bound the number of rational or torsion points independent of the geometry of the  curve, one is forced to work with primes of bad reduction. In addition, there are some curves $X/\Q$ for which the upper bound on $X(\Q)$ coming from Chabauty--Coleman method is achieved at primes of bad reduction (see, for example, \cite[Example~5.1]{katz_dzb13:chabauty_coleman}). Unfortunately, there are two different notions of $p$-adic integration: the abelian integral arising from the $p$-adic logarithm and the Berkovich--Coleman \cite{coleman_shalit88:p_adic_regulator,berkovic07:integration} integral performed locally on the curve; see Section~\ref{integrationtheories}. It is the abelian integral that is needed for applications, and it is the Berkovich--Coleman integral that can be computed. Specifically, to compute the Berkovich--Coleman integral along a path, one can cover the curve by {\em basic wide opens}, certain analytic open sets, each of which can be embedded into a good reduction curve. The path is broken into segments, each lying in a basic wide open. By picking a lift of Frobenius on each good reduction curve, and performing the Berkovich--Coleman integral there, one can compute the integral along each segment. This integral, however, may disagree with the abelian integral, and indeed, it may be path dependent. These issues arise because the two notions of integral differ on annuli. Indeed, the $1$-form $\omega$ develops poles when extended to the good reduction curve and one is forced to integrate logarithmic differentials of the form $dt/t$. To perform such an integral, one must pick a branch of $p$-adic logarithm. The ambiguity of this choice leads to the two notions of $p$-adic integration: a consistent choice of $p$-adic logarithm gives the Berkovich--Coleman integral; an integration method to force path independence gives the abelian integral. Fortunately, given the Berkovich--Coleman integral and some information about the reduction type of the curve, one can determine the abelian integral. Here, to compare the two integration theories, we follow \cite{krzb15:uniform_bounds} which makes use of the tropical Abel--Jacobi map. There are other approaches: Stoll \cite{stoll2018uniform} made a local analysis of the Abel--Jacobi map; Besser and Zerbes \cite{BesserZerbes} made use of $p$-adic height pairings \cite{BesserPairing}.

While $p$-adic integration on bad reduction curves has been used to prove theoretical results, the only available algorithms and examples have been on the Tate curve. Having algorithms to compute abelian integrals allows one to carry out the method of Chabauty--Coleman \cite{mccallum2012method} at primes of bad reduction. Moreover, additional refinements of such algorithms allow one to compute $p$-adic height pairings on curves \cite{BesserPairing} and $p$-adic regulators in $K$-theory \cite{BesserSyntReg}, again at primes of bad reduction. The purpose of this paper is to provide an algorithm for computing abelian integrals on bad reduction hyperelliptic curves for $p>2$.

Our algorithm works by first computing the Berkovich--Coleman integral $\lBCint \omega$ and then correcting it to an abelian integral. We consider a hyperelliptic curve interpreted as a map $\pi\colon X\to \P^1$ given by $y^2=f(x)$ for a polynomial $f(x)$ defined over a finite extension of $\Q_p$ for $p$ an odd prime. By examining the roots of $f(x)$ and making use of a Newton polygon argument, we are able to cover $\P^{1,\an}$ by open subspaces $\{U_i\}$ such that for each $U_i$ one can find a good reduction hyperelliptic curve $\tilde{X}_i$ into which $\pi^{-1}(U_i)$ embeds as the complement of  finitely many closed discs. In the process of finding the covering $\{\pi^{-1}(U_i)\}$, we determine the dual graph $\Gamma$ of the special fiber of a semistable model of $X$ and therefore its tropicalization.

We expand $\omega$ as a power series in certain meromorphic $1$-forms on $\tilde{X}_i$. We pick a set of meromorphic $1$-forms on $\tilde{X}_i$ that descend to a basis of the odd part of the de Rham cohomology of $\pi^{-1}(U_i)$. Then, by a pole-reduction argument similar to work of Tuitman \cite{Tuitman:counting1, Tuitman:counting2}, we rewrite the terms in the power series as the sum of an exact form and a linear combination of $1$-forms in our basis. Then, one is able to perform the integral using the techniques of \cite{BBKExplicit, BBComputing,balakrishnan2015coleman}. This allows us to integrate $\omega$ between points of $\pi^{-1}(U_i)$.

The Berkovich--Coleman integral is path dependent but invariant under fixed endpoint homotopy. The homotopy class of a path in $X^{\an}$ can be specified by its endpoints together with a path in $\Gamma$ between the tropicalizations of the endpoints. From our knowledge of the intersections of the $U_i$'s and thus of the dual graph, we are able to perform the Berkovich--Coleman integral along any path in $X^{\an}$. In particular, we can integrate $1$-forms along closed paths to determine the Berkovich--Coleman periods. Using them, together with a description of the tropical Abel--Jacobi map, we can correct the Berkovich--Coleman integral along a path to the abelian integral between its endpoints.

Our algorithm works in great generality; but implementing it in a computer algebra system seems out of reach at present, even if we take the base field to be $\Q_p$. There are two main reasons for this:
\begin{itemize}
    \item Sage includes implementations of the integration algorithms in \cite{BBKExplicit, BBComputing,balakrishnan2015coleman} when the base field is $\Q_p$. However, the hyperelliptic curves $\tilde{X}_i$ above are generally defined over non-trivial extensions of $\Q_p$.
    \item In our approach, it might be necessary to work with several extensions of $\Q_p$ at time same time. In Sage, Eisenstein and unramified extensions are implemented; however, neither conversion between these extensions nor general extensions are available.\footnotemark
\end{itemize}
We finally note that, when these obstacles are overcome, it should be possible to implement our algorithm in Sage.

\footnotetext[1]{Unfortunately, Magma has the same limitation.}

The paper is organized as follows. In section 2, we introduce some notation. Section 3 recalls Berkovich--Coleman and abelian integration and gives a formula for converting between them. Section 4 discusses coverings of $\P^1$ and of hyperelliptic curves. In section 5, we describe how to integrate a particular basis of $1$-forms on hyperelliptic basic wide opens. Section 6 provides a pole reduction argument that allows us to rewrite a $1$-form with poles as the sum of an exact form and a linear combination of our basis elements. The $1$-forms on the hyperelliptic curve are expanded as a power series on hyperelliptic basic wide opens in section 7. We compute Berkovich--Coleman integrals on paths and convert them into abelian integrals in section 8. Section 9 provides a number of examples and verifies the vanishing of certain abelian integrals between torsion points. 

\subsection*{Acknowledgements} We would like to acknowledge Jennifer Balakrishnan, Francesca Bianchi, Paul Helminck, Kiran Kedlaya, Daniel Litt and Joseph Rabinoff for helpful discussions regarding this work. The detailed comments of the referee were especially welcome. Special thanks are due to Steffen M\"{u}ller who encouraged our collaboration. Some of the work in this article was carried out during the second-named author's visit to the Ohio State University. He thanks the Mathematics cluster DIAMANT for partially supporting the visit. Eric Katz was partially supported by NSF DMS 1748837. Enis Kaya was partially supported by NWO grant 613.009.124.

\section{Preliminaries}
Let $p$ be an odd prime. Let $\C_p$ denote the completion of an algebraic closure of $\Q_p$. Let $v_p$ be the valuation on $\C_p$, normalized such that $v_p(p)=1$. It corresponds to the absolute value $\|\cdot\|_p$ where $\|a\|_p=p^{-v_p(a)}$. The field $\bar{\F}_p$, the algebraic closure of $\F_p$, is the residue field of $\C_p$. Let $\bK$ be a complete field of finite residue degree over $\Q_p$ with residue field $\bk$. Unless otherwise noted, $\bK$ will be a finite extension of $\Q_p$. Write $R$ for its valuation ring. 

\subsection{$p$-adic Analysis} 
In general, we will use the language of Coleman \cite{coleman89:reciprocity_laws} but will freely invoke Berkovich spaces when convenient. See \cite{krzb15:uniform_bounds} for more details.

We write 
\begin{eqnarray*}
B(a,r)&=&\{z\in\P^{1,\an} \mid \|z-a\|_p<r\},\\
B(\infty,r)&=&\{z\in\P^{1,\an} \mid \|1/z\|_p<r\}.
\end{eqnarray*}
We write $\overline{B}(a,r)$ for those sets when we replace the strict inequality with the nonstrict one. 

Analytic spaces are built by gluing affinoids. Given an affinoid of good reduction $V$, let $\red\colon V\to V_{\bk}$ be its reduction map. The preimage of a closed point under $\red$ is a \defi{residue disc}. In the case of $\P^{1,\an}$, the residue disc about $a\in\A^1(\C_p)$ is $B(a,1)$ while the residue disc about $\infty$ is $B(\infty,1)$. A map $\phi\colon V\to V$ is called \defi{a lift of Frobenius} if it induces the Frobenius map on $V_{\bk}$. 

\begin{defn}
A \defi{wide open} $U$ is a rigid analytic space isomorphic to the complement in a connected smooth complete curve $X$ of finitely many closed discs.
\end{defn}

\begin{defn} 
A \defi{basic wide open} $U$ is a rigid analytic space isomorphic to the complement in a connected good reduction complete curve of finitely many closed discs each contained in a distinct residue disc.
\end{defn}

Two types of basic wide opens will be used in this paper: rational and hyperelliptic. A basic wide open is called \defi{rational} (resp. \defi{hyperelliptic}) if it lies in the rigid analytic space associated to $\P^1$ (resp. a hyperelliptic curve). 

We have the following elementary examples of (rational) basic wide opens. A projective line with one closed disc removed is called an \defi{open disc}, and such a space is isomorphic to $B(0,1)$, the \defi{standard} open disc. Similarly, a projective line with two disjoint closed discs removed is called an \defi{open annulus}; such a space is isomorphic to a \defi{standard} open annulus, i.e., a space of the form 
\[A(r,1) = \{z\in\A^{1,\an}\mid r<\|z\|_p<1\},\ \ r<1.\]

\begin{defn} An \defi{underlying affinoid} of a basic wide open $U$ is an affinoid subdomain $V\subset U$ such that the connected components of $U\setminus V$ are annuli and are in bijective correspondence with the ends of $U$. These annuli are called \defi{boundary annuli}.
\end{defn}

Note that underlying affinoids are necessarily of good reduction. We will need to consider underlying affinoids within rational basic wide opens. Let $U$ be an rational basic wide open. For each closed disc $D_i$, we may pick a slightly larger open disc $D'_i$ (still contained in a residue disc) containing $D_i$. Then $V=U\setminus (\cup_{i} D'_i)$ is an underlying affinoid \cite[Corollary~3.5a]{coleman89:reciprocity_laws} and  $U\setminus V$ is a finite union of annuli. 

We recall some notions of analytic curves and their
skeletons~\cite{baker_payne_rabinoff13:analytic_curves}.
Let $X$ be a smooth, proper, geometrically connected $\bK$-curve.  Let $X^{\an}$ denote the Berkovich analytification of $X$ \cite{berkovic90:analytic_geometry}. Attached to a split semistable $R$-model $\fX$ of $X$ is a metric graph $\Gamma_{\fX}$ called its \defi{skeleton}. There is a retraction $\tau\colon X^\an\to\Gamma_\fX$. The vertices of $\Gamma_{\fX}$ correspond bijectively to the irreducible components of the special fiber of $\fX$. Any curve admits a split semistable model (and hence a skeleton) after 
making a finite extension of the ground field $\bK$. 

\subsection{Differential Forms}

Let $X/\bK$ be a curve of genus $g$.

\begin{defn}
A meromorphic $1$-form on $X$ over $\bK$ is said to be of the \defi{first kind} if it is holomorphic, of the \defi{second kind} if it has residue $0$ at every point, and of the \defi{third kind} if it is regular, except possibly for simple poles with integer residues.
\end{defn}

The exact differentials, i.e., differentials of rational functions, are of the second kind. The $\bK$-vector space of differentials of the second kind modulo exact differentials  is canonically isomorphic to $H_{\dR}^1(X/\bK)$, the first algebraic de Rham cohomology of $X/\bK$. 

\subsection{Hyperelliptic Curves} 
We will consider hyperelliptic curves defined by $y^2=f(x)$, for a polynomial $f(x)$ with distinct roots. The curve has a compactification $X$ with a degree $2$ map $\pi\colon X\to\P^1$. If $f(x)$ is of degree $d$, then $X$ is of genus $\lfloor \frac{d-1}{2}\rfloor$. 

The curve $X$ has a hyperelliptic involution extending $w(x,y)=(x,-y)$. The fixed points of the involution are the \defi{Weierstrass} points. If $d$ is even, then there are two distinct points lying over $\infty$, and these points are non-Weierstrass; if $d$ is odd, then there is a single point lying over $\infty$, and this point is Weierstrass.

On $X$, we say that a residue disc (with respect to $x$) is said to be \defi{Weierstrass} (resp. \defi{non-Weierstrass}) if it corresponds to a Weierstrass (resp. non-Weierstrass) point. In the odd degree case, we also distinguish between \defi{finite} and \defi{infinite} Weierstrass residue discs, which, respectively, correspond to finite and infinite Weierstrass points. 

\section{Berkovich--Coleman and Abelian Integration}
\label{integrationtheories}

We will define Berkovich--Coleman and abelian integration and give a formula for passing between them.

\subsection{$p$-adic Integration Theories}

In this subsection, we review $p$-adic integration theories, referring the reader to \cite{krzb15:uniform_bounds} for details.

Let $X$ be a smooth $\C_p$-analytic space, and let $\cP(X)$ be the set of paths
$\gamma\colon [0,1]\rightarrow X$ with ends in $X(\C_p)$.
Let $\Omega^1(X)$ be the space of holomorphic $1$-forms on $X$.   

\begin{defn}
An \defi{integration theory} on $X$ is a map $\int\colon \cP(X)\times \Omega^1(X)\rightarrow\C_p$ satisfying the following:
\begin{enumerate}
\item If $U\subset X$ is an open subdomain isomorphic to an open polydisc and $\omega|_U=df$ with $f$ analytic on $U$, then $\int_\gamma\omega=f(\gamma(1))-f(\gamma(0))$ for all $\gamma\colon [0,1]\to U$.
\item $\int_\gamma \omega$ only depends on the fixed endpoint homotopy class of $\gamma$.
\item If $\gamma'\in\cP(X)$ and $\gamma'(0)=\gamma(1)$, then
\[\int_{\gamma'*\gamma} \omega=\int_\gamma \omega+\int_{\gamma'} \omega \]
where $\gamma'*\gamma$ is the concatenation.
\item $\omega\mapsto\int_\gamma \omega$ is linear in $\omega$ for fixed $\gamma$.
\end{enumerate}
\end{defn}

One such integration theory is \defi{Berkovich--Coleman} integration $\lBCint$ which specifies a unique integral by fixing the integral on annuli and mandating a change-of-variables formula. Here, one fixes a branch of $p$-adic logarithm, $\Log$, and requires
\begin{enumerate}
\item if $X = \G_m^{\an}=\operatorname{Spec}(\C_p[t,t^{-1}])^{\an}$, then
\[\BCint_1^x \frac{dt}{t}=\Log(x),\]
\item if $h\colon X\rightarrow Y$ is a morphism, $\omega\in \Omega^1(Y)$ and $\gamma\in\cP(X)$, then
\[\BCint_\gamma h^*\omega=\BCint_{h(\gamma)} \omega.\] 
\end{enumerate} 

Unfortunately, the Berkovich--Coleman integral is generally path-dependent: $\lBCint_\gamma$ depends on $\gamma$ not just on its endpoints. However, when $X$ is simply-connected, $\lBCint_\gamma$ is path-independent by the homotopy invariance; in this case, we simply write $\lBCint_x^y=\lBCint_\gamma$ for any path $\gamma$ from $x$ to $y$.

For curves, uniqueness follows from covering the curve by basic wide opens and annuli. A path can be rewritten as a concatenation of paths, each staying in an element of the covering. The integral on basic wide opens is determined by the change-of-variables formula using a lift of Frobenius. The integral on annuli is determined by writing the restriction of the function as a two-sided
power series and integrating term-by-term (using $\Log(t)$ to antidifferentiate $dt/t$).

\begin{rem} \label{BCintegralislocal}
The Berkovich--Coleman integral is local in the sense that if $U\subset X$ is an analytic subdomain and $\gamma\colon [0,1]\to U$ is a path in $U$, the integral $\lBCint_\gamma\omega$ can be computed from $U$, $\gamma$, and $\omega|_U$. 
\end{rem}

The Berkovich--Coleman integration has a useful characterization on basic wide opens \cite[Section~2]{coleman_shalit88:p_adic_regulator}. On a basic wide open $U$, the Berkovich--Coleman integral is univalent: given $\omega\in\Omega^1(U)$, there is a locally analytic function $f_\omega$ unique up to a global constant such that
\[\BCint_\gamma \omega=f_\omega(\gamma(1))-f_\omega(\gamma(0)).\]
We describe such $f_\omega$. Fix a basic wide open $U$ and $\omega\in\Omega^1(U)$.  

Let $V$ be an underlying affinoid of $U$ and let $\cA$ denote the set of annuli which are the connected components of $U\setminus V$. A \defi{Frobenius neighborhood} of $V$ in $U$ is a pair consisting of a basic wide open $W$ with $V\subset W\subset U$ and a  morphism $\phi\colon W\to U$ restricting to a lift of Frobenius on $V$. 

For an open subdomain $U'\subset U$, let 
\[A_{\Log}(U')=A(U')[\{\Log(f)\mid f\in A(U')^{\times}\}]\]
where $A(U')$ is the set of analytic functions on $U'$. Write $\cL(U)$ for locally analytic functions on $U$. Note that
\[A_{\Log}(E)\subset\cL(E) \text{ for each } E\in\cA.\] 
Let $P(T)$ be a polynomial without roots-of-unity roots such that (after possibly shrinking $W$ to ensure $P(\phi^*)$ is well-defined) $P(\phi^*)$ annihilates $(\Omega^1(U)|_W)/dA(W)$. Such a polynomial exists by the Weil conjectures for curves.

\begin{lem} \emph{(\cite[Proposition~2.4.1]{coleman_shalit88:p_adic_regulator})}
The locally analytic function $f_\omega\in\cL(U)$ is characterized up to addition of a global constant by the following properties:
\begin{enumerate}
    \item $f_\omega|_E \in A_{\Log}(E)$ for each $E\in\cA$,
    \item $df_\omega=\omega$, and
    \item $P(\phi^*)f_\omega|_W\in A(W)$. 
\end{enumerate}
\end{lem}

Functions that satisfy the properties of the above lemma are said to be \defi{Coleman analytic} or to be a \defi{Coleman primitive} of $\omega$.

The following result which interchanges limits and integration follows from applying the characterization to $\lim f_{\omega_i}$.

\begin{prop} \label{swapintsum} 
For a basic wide open $U$ and an underlying affinoid $V$ in $U$, let $\cA$, $W$ and $\phi$ be as above. If $\{f_{\omega_i}\}$ be a sequence of locally uniformly convergent Coleman analytic functions on $U$ such that
\begin{enumerate}
    \item $\{f_{\omega_i}|_E\}$ converges uniformly in $A_{\Log}(E)$ for each $E\in\cA$.
    \item $\{\omega_i\}$ converges uniformly in $\Omega^1(U)$, and
    \item $\{P(\phi^*)f_{\omega_i}|_W\}$ converges uniformly on $W$.
\end{enumerate}
Then the locally analytic limit $\lim f_{\omega_i}$ is a Coleman primitive of $\lim \omega_i$ on $U$.
\end{prop}

Another approach to defining a $p$-adic integration theory on
a curve is via the $p$-adic Lie theory of its Jacobian.  This was done in great generality by Zarhin~\cite{Zarhin1996} and Colmez \cite{colmez1998integration}.

Let $A$ be an abelian variety over $\C_p$. Recall that every $1$-form on $A$ is translation-invariant. In other words, \[\Omega^1(A) = \Omega^1_{\operatorname{inv}}(A).\]
The \defi{abelian logarithm} on $A$ is the unique homomorphism of $\C_p$-Lie groups 
\[\log_{A(\C_p)}\colon  A(\C_p)\to\Lie(A)\] 
whose linearization
\[d\log_{A(\C_p)} \colon \Lie(A) \To \Lie(\Lie(A)) = \Lie(A)
\]
is the identity map. See~\cite{Zarhin1996} for the existence and uniqueness of
$\log_{A(\C_p)}$.  For $x\in A(\C_p)$ and $\omega\in\Omega^1(A)$, we define
\[ \presuper\Ab\int_0^x\omega = \angles{\log_{A(\C_p)}(x),\,\omega}\]
where $\angles{\scdot,\scdot}$ is the pairing between $\Lie(A)$ and
$\Omega^1(A)$.  For $x,y\in A(\C_p)$, we set
\[ \presuper\Ab\int_x^y\omega = \presuper\Ab\int_0^y\omega -
\presuper\Ab\int_0^x\omega. \]
We call $\lAbint$ the \defi{abelian integral} on $A$.

The abelian logarithm and the abelian integral are functorial under homomorphisms of abelian varieties: if $h\colon A\rightarrow B$ is a homomorphism, then
\begin{enumerate}
    \item $dh\circ \log_{A(\C_p)} = \log_{B(\C_p)}\circ\ h$,
    \item for $\omega\in\Omega^1(B)$ and $x,y\in A(\C_p)$,
    \[\presuper{\Ab}{\int}_x^y h^*\omega=\presuper{\Ab}{\int}_{h(x)}^{h(y)}\omega.\]
\end{enumerate}

We may define an integration theory on a smooth, proper, connected curve $X$ over $\C_p$ by pulling back the abelian integral from its Jacobian $J$ by the Abel--Jacobi map $\iota\colon X\to J$ with respect to a base-point $x_0\in X(\C_p)$. This integral depends only on the endpoints of a path $\gamma$, but it is not local.

\begin{rem} \label{torsionvanish}
Because $\Lie(J)$ is torsion free, if $x,y$ are points in $X(\C_p)$ such that $[y]-[x]$ represents a torsion point of $J(\C_p)$, then 
\[\Abint_x^y \omega = 0 \text{ for all } \omega\in \Omega^1(X).\]
\end{rem}

\subsection{Integral Comparison}

Following \cite{krzb15:uniform_bounds}, we can compare the Berkovich--Coleman and abelian integrals. Let $A$ be an abelian variety over $\C_p$ and let $\pi\colon  E^\an\to A^\an$ be the topological universal cover of $A^\an$. We have the \defi{Raynaud uniformization cross},
\[\xymatrix @=.20in{
    & {T^\an} \ar[d] & \\
    {M'} \ar[r] & {E^\an} \ar[r]^\pi\ar[d] & {A^\an} \\
    & {B^\an} &}\]
with exact row and column where $M'$ is canonically isomorphic to $\pi_1(A^\an)=H_1(A^{\an};\Z)$, $T$ is a torus and $B$ is an abelian variety with good reduction.
Let $M$ be the character lattice of $T$, so $T = \Spec(\C_p[M])$.

Let $N = \Hom(M,\Z)$. There is a surjective group homomorphism, the \defi{tropicalization map} $\trop\colon  E(\C_p)\to N_\Q=\Hom(M,\Q)$. The restriction of
$\trop$ to $M'\subset E(\C_p)$ is injective, and its image
$\trop(M')\subset N_\Q$ is a full-rank lattice in the real vector space
$N_\R=\Hom(M,\R)$.  We can define the real torus $\Sigma=N_\R/\trop(M')$ to be the \defi{skeleton} of $A$. 
The tropicalization map $\bar\tau\colon A^{\an}\to\Sigma$ is defined as the quotient of $\trop$ and fits into the following
commutative diagram:
\[\xymatrix @R=.2in{
    0 \ar[r] & {M'} \ar[r] \ar[d]^{\trop}_\simeq
    & {E^\an} \ar[r] \ar[d]^{\trop}
    & {A^\an} \ar[r] \ar[d]^{\bar\tau} & 0 \\
    0 \ar[r] & {\trop(M')} \ar[r] 
    & {N_\R} \ar[r] & {\Sigma} \ar[r] & 0}\]
The torus $\Sigma$ is a deformation retract of $A^{\an}$.

To compare the two integrals, we first define logarithms $\Log_{\BC},\Log_{\Ab}\colon E(\C_p)\to \Lie(E)$. Using the isomorphisms, 
\[\Lie(E)\cong \Omega^1_{\iinv}(E)^*\cong \Omega^1_{\iinv}(A)^*\cong \Omega^1(A)^*,\]
we define
\begin{align*}
\Log_{\BC}\colon E(\C_p)&\to \Lie(E) & \Log_{\Ab}\colon E(\C_p)&\to \Lie(E)\\
x&\mapsto \left[\omega\mapsto \BCint_0^x \omega\right]& x&\mapsto \left[\omega\mapsto \Abint_0^{\pi(x)} \omega\right].
\end{align*}

\begin{prop}
\emph{(\cite[Proposition~3.16]{krzb15:uniform_bounds})} The difference between the two logarithms
\[\Log_{\BC}-\Log_{\Ab}\colon E(\C_p)\to \Lie(E)\]
factors as
\[\xymatrix{
E(\C_p)\ar[r]^{\trop}&N_{\Q}\ar[r]^>>>>L&\Lie(E)}\]
where $L$ is a linear map.
\end{prop}

Using the identification $H_1(A^{\an};\Z)\cong M' \cong \trop(M')$ and the inclusion $\trop(M')\subset N_\Q$, we have the following: 
\begin{lem} \label{FormulaForL}
The map $L$ is characterized by the property that for any $C\in H_1(A^{\an};\Z)$,
\[L(C)=\left[\omega\mapsto \BCint_{\gamma} \omega\right]\]
where $\gamma$ is any loop in $\cP(A^{\an})$ whose homology class is equal to $C$.
\end{lem}

\begin{proof}
Because the abelian logarithm is defined on $A(\C_p)$ (not just its universal cover $E(\C_p)$), we see that $\Log_{\Ab}(M')=0.$ Consequently, $L(C)=\Log_{\BC}(\tilde{\gamma}(1))$
where $\tilde{\gamma}$ is the lift of $\gamma$ in $E^{\an}$ based at the identity element in $E^{\an}(\C_p)$.
\end{proof}

\subsection{Tropical Integration and The Comparison Formula}

We will need to pull back the comparison between integrals to a curve $X$ via its Abel--Jacobi map $\iota\colon X\to J$. To do so, we will make use of the {\em tropical Abel--Jacobi map} which was described using tropical integration by Mikhalkin--Zharkov \cite{mikhalkin_zharkov08:tropical_curves} (see also \cite[Section~3]{baker_rabinoff14:skeleton_jacobian}) together with some results of Baker--Rabinoff \cite{baker_rabinoff14:skeleton_jacobian}. The statement of the comparison 
result is different from that given in \cite{krzb15:uniform_bounds}.

Let $\Gamma$ be a finite connected graph (usually taken to be a graph structure on the skeleton of a curve $X^{\an}$). We will parameterize each oriented edge $e=vw$  by $[0,1]$ using the coordinate $t$ such that $v$ corresponds to $t=0$ and $w$ corresponds to $t=1$. By flipping the orientation of the edge, we change the parameterization by $t'=1-t$. We take each edge of $\Gamma$ to be of length $1$.

\begin{defn}
A \defi{tropical $1$-form} on $\Gamma$ is a function $a\colon \vec{E}(\Gamma)\to\R$ from the set of directed edges to the real numbers such that
\begin{enumerate}
    \item $a(\overline{e})=-a(e)$ where $\overline{e}$ is $e$ with the reversed orientation, and
    \item $a$ satisfies the \defi{harmonicity condition}: for each $v\in V(\Gamma)$,
\[\sum_e a(e)=0\]
where the sum is over edges adjacent to $v$ directed away from $v$.
\end{enumerate} 
Denote the space of tropical $1$-forms on $\Gamma$ by $\Omega^1_{\trop}(\Gamma)$.
\end{defn}

To an oriented edge $e=vw$ of $\Gamma$, let $\eta_e$ be the function $\vec{E}(\Gamma)\to \R$ that is $0$ away from $e$ and takes the value $1$ on $e$ with the given orientation (and $-1$ on $\overline{e}$). For a cycle $C=\sum_e a_e e\in H_1(\Gamma;\R)$, define 
\[\eta_C=\sum_e a_e \eta_e.\]
It is easily seen that $\eta_C$ is a tropical $1$-form.

Given a path $\gamma$ specified as a sequence of directed edges $\gamma=e_1e_2\dots e_\ell$, we define the \defi{tropical integral} of a tropical $1$-form $\eta$ on $\gamma$ by
\[\tint_{\gamma} \eta \coloneq \sum_{i=1}^{\ell} \eta(e_i).\]
Moreover, we may extend the tropical integral to paths between points on $\Gamma$. To a path between points $p$ and $q$ contained in an edge $e$, we define
\[\tint_p^q \eta \coloneq \eta(e)(q-p)\]
where we identify $e$ with $[0,1]$ by use of the orientation on $e$. Then, we extend tropical integration to arbitrary paths by additivity of integrals under concatenation of paths.

For a closed path $\gamma$, this integral is seen to only depend on $[\gamma]\in H_1(\Gamma;\R)$.
Therefore, tropical integration gives a map 
\begin{align*} 
\mu\colon H_1(\Gamma;\R)&\to \Omega^1_{\trop}(\Gamma)^*\\
C&\mapsto \left[\eta\mapsto \tint_{C} \eta\right].
\end{align*}

Recall that the cycle pairing $\langle\cdot,\cdot\rangle$ on $H_1(\Gamma;\R)\subset C_1(\Gamma;\R)$ is the pairing induced from the inner product on $C_1(\Gamma;\R)$ making the set of edges (oriented in some way) into an orthonormal basis. In other words, this pairing takes cycles $C$ and $D$ to the length of their oriented intersection.

The following is easily verified.
\begin{prop} Tropical integration is equal to the cycle pairing in the following sense: for $C,D\in H_1(\Gamma;\R)$,
\[\tint_C \eta_D=\langle C,D\rangle.\]
\end{prop}

This proposition implies that the map $\mu$ is an isomorphism because the cycle pairing is nondegenerate on $H_1(\Gamma;\R)$. 

\begin{cor}
Any cohomology class in $H^1(\Gamma;\R)$ can be represented by a tropical $1$-form: for any $c\in H^1(\Gamma;\R)$, there is a tropical $1$-form $\eta$ such that
\[c(D)=\tint_D \eta\] 
for any $D\in H_1(\Gamma;\R)$.
\end{cor}

\begin{cor} \label{l:dualbasis} There exists a basis $C_1,\dots,C_h$ of $H_1(\Gamma;\R)$ and a basis $\eta_1,\dots,\eta_h$ of $\Omega^1_{\trop}(\Gamma)$ such that 
\[
\tint_{C_i} \eta_j=
\begin{cases}
1 & \text{if }\ i=j,  \\
0 & \text{if }\ i\neq j.
\end{cases}
\]
\end{cor} 

Now, let $X$ be a smooth, proper, connected curve over $\C_p$ with skeleton $\Gamma$. Note that $\Gamma$ is a deformation retract of $X^{\an}$. Let $J$ be its Jacobian. We can identify the real torus $\Sigma$ from $\Gamma$. 

\begin{prop} There is an isomorphism of real tori 
\[\Sigma\cong \Omega^1_{\trop}(\Gamma)^*/\mu(H_1(\Gamma;\Z)).\]
\end{prop}

\begin{proof}
Because the Abel--Jacobi map $\iota$ induces an isomorphism $\iota_*\colon H_1(X^{\an};\R)\to H_1(J^{\an};\R)$, we have a sequence of isomorphisms,
\[N_{\R}\cong\trop(M')\otimes \R\cong H_1(J^{\an};\R)\cong H_1(X^{\an};\R)\cong H_1(\Gamma;\R)\cong \Omega^1_{\trop}(\Gamma)^*.\]
Under the composition, $\trop(M')$ is mapped to $\mu(H_1(\Gamma;\Z)).$
\end{proof}
Let $P_0$ be a base-point of $\Gamma$. Now, if we let $\tilde{\Gamma}$ denote the universal cover of $\Gamma$ with a base-point $\tilde{P}_0$ over $P_0$, tropical integration gives a map
\begin{align*} 
\tilde{\beta}\colon \tilde{\Gamma}&\to \Omega^1_{\trop}(\Gamma)^*\\
\tilde{Q}&\mapsto \left[\eta\mapsto \tint_{\tilde{P}_0}^{\tilde{Q}} \eta\coloneq\tint_{\gamma} \eta\right]
\end{align*}
where $\gamma$ is the image in $\Gamma$ of the unique path in $\tilde{\Gamma}$ from $\tilde{P}_0$ to $\tilde{Q}$. The map $\tilde{\beta}$ descends to quotients giving the \defi{tropical Abel--Jacobi map}
\[\beta\colon \Gamma\to \Omega^1_{\trop}(\Gamma)^*/\mu(H_1(\Gamma;\Z))\cong\Sigma.\]
The tropical Abel--Jacobi map map is equal to the tropicalization of the Abel--Jacobi map in the following sense. Let $\iota\colon X\to J$ be the  Abel--Jacobi map with respect to $x_0\in X(\C_p)$ . 
Let  $\tau\colon X^{\an}\to \Gamma$ be the tropicalization map and set $P_0=\tau(x_0)$. By a result of Baker--Rabinoff \cite[Proposition~6.1]{baker_rabinoff14:skeleton_jacobian}, the following diagram commutes:
  \[\xymatrix @=.2in{
  {X^\an} \ar[r]^\iota \ar[d]_\tau & {J^\an} \ar[d]^{\bar\tau} \\
  {\Gamma} \ar[r]_\beta & {\Sigma}.
  }\]

Now, we can give a comparison theorem for Berkovich--Coleman and abelian integrals.

\begin{thm} \label{thm:abintformula} Let $X^{\an}$ be a connected, smooth, compact analytic curve over $\C_p$ with skeleton $\Gamma$ and retraction $\tau\colon X^{\an}\to \Gamma$. Let $x_0\in X(\C_p)$ be a base-point and set $P_0=\tau(x_0)$. Let $C_1,\dots,C_h$ and $\eta_1,\dots,\eta_h$ be as in Corollary~\ref{l:dualbasis}. Let $\gamma_1,\dots,\gamma_h$ be loops in $X^{\an}$ whose homology classes are $C_1,\dots,C_h$, respectively.
The following formula holds: for $x\in X(\C_p)$, pick a path $\gamma$ in $X^{\an}$ with $\gamma(0)=x_0$ and $\gamma(1)=x$, then
\[\BCint_\gamma\omega-\Abint_{x_0}^x\omega=\sum_i \left(\BCint_{\gamma_i} \omega\right)\left(\tint_{\tau(\gamma)} \eta_i\right)\]
for every holomorphic $1$-form $\omega$. 
\end{thm}

\begin{proof}
 Let $\widetilde{X^{\an}}$ be the topological universal cover of $X^{\an}$. We have a commutative diagram
 \[\xymatrix{
\widetilde{X^{\an}}(\C_p)\ar[r]^>>>>{\tilde{\iota}}\ar[d]^{\tau}&\widetilde{J^{\an}}(\C_p)=E(\C_p)\ar[d]^{\bar{\tau}}\ar[rrr]^>>>>>>>>>>>>>{\Log_{\BC}-\Log_{\Ab}}&&&\Lie(E)\\
\tilde{\Gamma}\ar[r]^{\tilde{\beta}}&\tilde{\Sigma}=N_\R\ar[urrr]^L&&&\ 
}
\]
Now, consider the image of the lift $\tilde{\gamma}(1)$ of $x$ under the maps in the top row, evaluated on $\omega$. It suffices to show that $L\circ\tilde{\beta}\colon \tilde{\Gamma}\to \Lie(E)=\Omega^1(X^{\an})^*$ is given by
\[\tilde{Q} \mapsto \left[\omega\mapsto \sum_i \left(\BCint_{\gamma_i} \omega\right)\left(\tint_{\tilde{P}_0}^{\tilde{Q}} \eta_i\right)\right].\]
Under the identification $\Omega_{\trop}^1(\Gamma)^*\cong H_1(\Gamma;\R)$,
we claim that for $\tilde{Q}\in\tilde{\Gamma}$, we have
\[\tilde{\beta}(\tilde{Q})=\sum_i \left(\tint_{\tilde{P}_0}^{\tilde{Q}} \eta_i\right) C_i.\]
This is true after evaluating by $\eta_j\in \Omega^1_{\trop}(\Gamma)\cong H^1(\Gamma;\R)\cong H_1(\Sigma;\R)^*\cong M\otimes\R$:
\[\eta_j({\tilde{\beta}(\tilde{Q}))} =\tint_{\tilde{P}_0}^{\tilde{Q}} \eta_j
=\sum_i \left(\tint_{\tilde{P}_0}^{\tilde{Q}} \eta_i\right) \eta_j(C_i).\]
Because the $\eta_j$'s form a basis for $\Omega^1_{\trop}(\Gamma)$, the claim follows.

Applying $L$, we see 
\begin{eqnarray*}
L(\tilde{\beta}(\tilde{Q}))&=&
\left[\omega\mapsto \sum_i \Bigg(\tint_{\tilde{P}}^{\tilde{Q}} \eta_i\Bigg) \Bigg(L(C_i)(\omega)\Bigg)\right]\\
&=&
\left[\omega\mapsto \sum_i \Bigg(\tint_{\tilde{P}}^{\tilde{Q}} \eta_i\Bigg) \Bigg(\BCint_{\gamma_i}\omega\Bigg)\right]\\ 
\end{eqnarray*} by Lemma~\ref{FormulaForL}.
\end{proof}

\section{Coverings of Curves}

\subsection{Semistable Coverings}

Given a finite set of $\C_p$-points $S$ on $\P^1$, we will define a covering of $\P^{1,\an}$ by rational basic wide opens with respect to $S$. This will allow us to define a covering of the hyperelliptic curve $y^2=f(x)$ by hyperelliptic basic wide opens when we set $S$ to be the {\em roots} of $f(x)$. 

We follow \cite{coleman_shalit88:p_adic_regulator,CMStableReduction} in using the notion of semistable covering. 

\begin{defn} Let $Y$ be a smooth, compact, connected analytic curve over $\C_p$. A \defi{covering} $\cC$ of $Y$
is an admissible finite covering by distinct wide open subspaces of $Y$. The \defi{dual graph} $\Gamma(\cC)$ of the covering is a finite graph whose vertices correspond to elements of $\cC$ such that the edges between $U$ and $V$ correspond to components of $U\cap V$ while the self-edges at $U$ correspond to ordinary double-points in the reduction of $U$. The covering is said to be \defi{semistable} if, in addition,
\begin{enumerate}
    \item  If $U,V,W\in \cC$ then $U$ is disconnected from every component of $V\cap W$,
    \item If $U\in \cC$ then $U^\circ=U\setminus \bigcup_{V\neq U} V$ is a non-empty affinoid subdomain in $U$ whose reduction $U^\circ_{\bk}$ is absolutely irreducible, reduced, and has no singularities except ordinary double-points, and
    \item The genus of $Y$ obeys
    \[g(Y)=\sum_{U\in \cC} g(U^{\circ}_{\bk}) + b_1(\Gamma(\cC))\]
    where $b_1(\Gamma(\cC))$ is the first Betti number of $\Gamma(\cC)$. \label{i:genuscondition}
\end{enumerate}
We say that an element $U\in \cC$ is \defi{good} with respect to a subset $S\subset Y(\C_p)$ if there is an embedding into a compact good reduction curve, $\iota\colon U\to Y_U$ such that the points of $\iota(S\cap U(\C_p))$ lie in distinct residue discs. We say $\cC$ is \defi{good} with respect to $S$ if each element of $S$ belongs to at most one element of $\cC$ and each $U\in \cC$ is good with respect to $S$. The dual graph $\Gamma(\cC,S)$ of the covering with respect to $S$ is obtained from $\Gamma(\cC)$ by attaching half-open edges corresponding to elements of $S$ to the vertices corresponding to the elements of $\cC$ containing them.
\end{defn}

\subsection{Rational Coverings}

We discuss the existence of good semistable coverings of $\P^{1,\an}$ with respect to a given set of points $S\subset \P^1(\C_p)$. If $S$ contains only one element, then this is clear. From now on, let us assume that $S$ has at least two elements.

\begin{thm} \label{t:goodcover} Let $S\subset \P^1(\C_p)$ be a finite set. There is a good semistable covering $\cC$ of $\P^{1,\an}$ with respect to $S$. 
\end{thm}

We will prove Theorem~\ref{t:goodcover} by an inductive argument making use of Lemma~\ref{t:goodcoverlemma}. We will produce a dual graph attached to the covering as we proceed. To do so, we introduce semistable coverings of open discs by rational basic wide opens. They are defined as above except condition \eqref{i:genuscondition} is replaced by the condition
\[g(Y)=\sum_{U\in \cC} g(U) + b_1(\Gamma(\cC))\]
where $g(U)=0$ for a rational basic wide open $U$. This mandates that $\Gamma(\cC)$ is a tree. The embeddings $\iota\colon U\to \P^{1,\an}$ will be linear fractional transformations.

\begin{lem} \label{t:goodcoverlemma}
Let $R\in \|\C_p^*\|_p$ and $\beta\in\A^1(\C_p)$. Set $Y=B(\beta,R)$. For any non-empty finite subset $S$ of $Y(\C_p)$, there is a good semistable covering $\cC_Y$ of $Y$ with respect to $S$. The dual graph of the covering respect to $S$ is a rooted tree $T_Y$.
\end{lem}

\begin{proof}
Write $S=\{\alpha_1,\dots,\alpha_n\}$.
By translating by $-\alpha_1$, we may assume that $\alpha_1=0$. If the set $S$ has at least two elements, by scaling, we may assume that $\max_{i\neq 1}(\|\alpha_i\|_p)=1$.

We induct on $n$. If $n\leq 2$, then all the points of $S$ are in distinct residue discs, and we may let $\cC_Y=\{Y\}$. The tree attached to this covering is a single vertex with a half-open edge for each element of $S$.

Let $n>2$. Not all elements of $S$ are in a single residue disc. Let $I_1,\dots,I_m$ be the partition of $S$ according to which residue disc a point belongs. For each $i$ such that $|I_i|\geq 2$ pick a point $\beta_i\in I_i$. Let $R_i$ be the largest element of $\|\C_p^*\|_p$ such that $B(\beta_i,R_i)\cap S=I_i$ (so that $\overline{B}(\beta_i,R_i)$ contains some point of $S\setminus I_i$). Set $Y_i=B(\beta_i,R_i)$ and $S_i=I_i$. Because $|I_i|<n$, $Y_i$ has a good semistable covering $\cC_{Y_i}$ with respect to $S_i$. Now, let 
\[U=Y\setminus \bigsqcup_{i: |I_i|\geq 2} \overline{B}(\beta_i,r_i)\] 
where $r_i$ is the smallest element of $\|\C_p^*\|_p$ such that $\overline{B}(\beta_i,r_i)\cap S=I_i$. The covering $\cC_Y$ is defined as 
\[\{U\}\cup \bigcup_{i: |I_i|\geq 2} \cC_{Y_i}.\] 
Because there is at most one element of $S$ in every residue disc of $U$, $U$ is good with respect to $S$. Moreover, 
since each element of $S$ is either contained in $U$ or in exactly one element of the covering $\cC_{Y_i}$ for some $i$, the covering 
$\cC_Y$ is good with respect to $S$.

Denote the rooted tree corresponding to the covering $\cC_{Y_i}$ by $T_{Y_i}$. Consider the tree whose root is $U$ and where $U$ is connected to the roots of $T_{Y_i}$ for each $i$ with $|I_i|\geq 2$. To obtain $T_Y$ from this tree, attach to $U$ half-open edges corresponding to $I_i$ with $|I_i|=1$. These half-open edges correspond to the points of $S$ that are contained in $U$. 
\end{proof}

\begin{proof}[Proof of Theorem~\ref{t:goodcover}] 
Write $S=\{\alpha_1,\dots,\alpha_n\}$. 
Let $S'$ be the elements of $S$ contained in $\A^1(\C_p)$, and let $r$ be the maximum of their $p$-adic absolute values.
Pick $R\in \|\C_p^*\|_p$ with $R>r$ and set $Y=B(0,R)$. Using Lemma~\ref{t:goodcoverlemma}, find a good semistable covering $\cC_Y$ of $Y$ with respect to $S'$ and its rooted tree $T_Y$. This covering together with $U=\P^{1,\an}\setminus \overline{B}(0,r)$ is our desired covering $\cC$. 

The dual graph $\Gamma(\cC,S)$ is obtained by adjoining the vertex corresponding to $U$ to the root of $T_Y$ and then attaching the half-open edge corresponding to $\infty$ to $U$ if $\infty\in S$.
\end{proof}

One can see from general considerations or by examining the above construction that the intersection of two distinct elements of the semistable cover is either empty or an annulus.

Notice that the dual graph $\Gamma(\cC,S)$ is also a tree. From now on, we will denote this tree by $T$.

\begin{rem} The covering and graph can be constructed intrinsically using Berkovich spaces \cite{baker_payne_rabinoff13:analytic_curves}.
The tree $T$ consists of the type $II$ and type $III$ points of $\P^{1,\an}$ corresponding to discs of the form $\overline{B}(\alpha,r)$ 
for $\alpha\in \P^1(\C_p)$ such that at least one point of $S$ is contained in each of $\overline{B}(\alpha,r)$ and $\P^{1,\an}\setminus B(\alpha,r)$.

The covering can be obtained by taking a semistable vertex set as in \cite{baker_payne_rabinoff13:analytic_curves}. It consists of the type $II$ points of the form $\overline{B}(\beta_i,r_i)$ as constructed in the above algorithm.
\end{rem}

The following is straightforward:

\begin{lem} The dual graph $T$ is a graph structure on the skeleton of $\P^{1,\an}\setminus S$.
\end{lem}

\begin{rem} \label{rmk:coveringtoskeleton}
The above lemma can be proved in several ways. One can use the semistable vertex set to identify the skeleton as in \cite[Section~3]{baker_payne_rabinoff13:analytic_curves}. Alternatively, one can construct a semistable model from the semistable covering \cite[Theorem~1.2]{CMStableReduction} and obtain the skeleton by \cite[Section~4]{baker_payne_rabinoff13:analytic_curves}.
\end{rem}

Let $f(x)$ be a non-constant polynomial with coefficients in $\bK$. We define the \defi{roots} of $f(x)$ to be the usual zeroes of $f(x)$ together with $\infty$ if $f(x)$ has odd degree and write the set of roots as $S_f$. Lemma~\ref{t:goodcoverlemma} and Theorem~\ref{t:goodcover} can be turned into an algorithm for constructing a good covering of $\P^{1,\an}$ with respect to $S_f$. A priori, it looks as it would be necessary to exactly solve for the roots of $f(x)$. However, this can be avoided. First, we can approximate a root and use a translation to put it in $B(0,1)$. Then, we can find the valuation of the roots by using the theory of Newton polygons. We can rescale $x$ by an element of $\C_p$ to make sure that the largest absolute value of the roots is $1$. Indeed, the polynomial $f(x)$ has a root of valuation $s$ if and only if its Newton polygon has a segment of slope $-s$. This segment corresponds to roots of $p$-adic absolute value equal to $p^{-s}$.  After these reduction steps, the methods in the proof of Lemma~\ref{t:goodcoverlemma} are still applicable. 

\smallskip

\begin{algorithm}[H] \label{coversofP1}

\caption{\bf Covering of $B(0,R)$ with respect to the roots of a polynomial}

\KwIn{
\begin{itemize}
    \item A polynomial $f\in \bK[x]$ of degree at least $2$ whose finite roots are in $\bK$.
    \item A disc $B(0,R)$. 
\end{itemize}
}

\KwOut{A good semistable covering $\cC$ of $B(0,R)$ with respect to the roots of $f(x)$ contained in $B(0,R)$ together with its rooted tree $T$.}

\begin{enumerate}
    \item Find $\alpha\in B(0,R)$ sufficiently close to a root of $f(x)$ such that the polynomial $f(x-\alpha)$ has roots of at least two different $p$-adic absolute values in $B(0,R)$.\footnotemark~ Replace $f(x)$ by $f(x-\alpha)$. 
    \item Compute the slopes of the Newton polygon of $f(x)$ and set $\lambda$ to be the maximum of the slopes less than or equal to $\log_p(R)$.
    \item Pick $c\in\bK$ satisfying $v_p(c)=-\lambda$. Replace $f(x)$ by $f(cx)$ and $R$ by $\frac{R}{p^{\lambda}}$
    so that the maximum of $p$-adic absolute value of roots in the disc is $1$.
    \item Multiply $f(x)$ by a power of a uniformizer of $K$ to ensure that the minimum of the valuation of the coefficients of $f(x)$ is $0$. Factor the polynomial $f(x)\ \text{mod}\ p$ to determine the partition $\{I_i\}$ of the set of roots of $f(x)$ in $B(0,R)$ according to which residue disc the root belongs.
    \item For each $i$ with $|I_i|\geq 2$, pick a point $\beta_i$ in the same residue disc as the points in $I_i$. Set $f_i(x) = f(x-\beta_i)$. Let $\lambda_i$ be the largest negative slope of the Newton polygon of $f_i$, and let $\Lambda_i$ be the smallest positive slope of the Newton polygon of $f_i$. Set $r_i=p^{\lambda_i}$ and $R_i=p^{\Lambda_i}$. Apply this algorithm to $f_i(x)$ and the disc $B(0,R_i)$ to find a good covering $\cC_i$ of $B(\beta_i,R_i)$ with respect to $I_i$ together with its rooted tree $T_i$.
    \item Set $U=B(0,R)\setminus \bigcup_{i: |I_i|\geq 2} \overline{B}(\beta_i,r_i)$. Combine $U$ with the coverings $\cC_i$ found in the previous step to obtain the covering $\cC$ of $B(0,R)$.
    \item Let $U$ be the root of the tree $T$; for each $i$, do the following: if $|I_i|\geq 2$, attach the root of $T_i$ to $U$ by an edge; if $|I_i|=1$, attach a half-open edge corresponding to the unique element of $I_i$.
    \item Return $\cC$ and $T$.
    \end{enumerate} 
\end{algorithm} 

\footnotetext[2]{This can be verified by checking that the Newton polygon of $f(x-\alpha)$ has at least two segments of distinct slope less than or equal to $\log_p(R)$.} 

\smallskip

Algorithm~\ref{coversofP1} produces a good semistable covering of $\P^{1,\an}$ with respect to the roots of $f(x)$, by following the proof of Theorem~\ref{t:goodcover}. 

\subsection{Hyperelliptic Coverings} \label{HyperellipticCoverings}

Let $\pi:X^\an\rightarrow \P^{1,\an}$ be the analytification of the proper hyperelliptic curve defined by $y^2=f(x)$. Using ideas similar to those of Stoll in \cite{stoll2018uniform}, we will show that any good semistable covering of $\P^{1,\an}$ with respect to the roots of $f(x)$ induces a good semistable covering of $X^{\an}$ with respect to the Weierstrass points by taking inverse images. This covering will have a nice combinatorial structure whose dual graph $\Gamma$ is a  double cover of the dual graph $T$ of the covering of $\P^{1,\an}$. 

For a polynomial $f(x)$, write the roots as $S_f=\{\alpha_1,\alpha_2,\dots,\alpha_n\}$. Recall that, if $f(x)$ is of odd degree, we follow the convention of counting $\infty$ as one of its roots. In any case, $f(x)$ has an even number of roots. If $f(x)$ has distinct roots, the curve $y^2=f(x)$ has a compactification as a nonsingular curve $X$. 
If the roots of $f(x)$ are distinct mod $p$, then $X$ has good reduction over some extension of $\Q_p$. 

We will need the following observation:
\begin{lem} \label{squarerootsofseries}
Let $R\in \|\bK^*\|_p$. Let $h(x)\in \bK[x]$ be a polynomial. Let $D_1,\dots,D_m$ be disjoint closed discs in $\overline{B}(0,R)$ the union of whose interior contains $S_h\cap \overline{B}(0,R)$. Suppose that an even number of roots of $h(x)$  is contained in each $D_i$.  Then $h(x)$ has an analytic square root on $B(0,R)\setminus \left(\bigcup_i D_i\right).$
\end{lem}

\begin{proof}
We work with one disc of $D_1,\dots,D_m$ at a time, beginning with a disc $D$. Write $D=\overline{B}(\gamma,r)$. Let  
$\alpha_1,\dots,\alpha_{2\ell}$ be the roots of $h(x)$ contained in $D$. Then,
\[
\bigg(\prod_{i=1}^{2\ell}(x-\alpha_i)\bigg)^{1/2} = (x-\gamma)^{\ell}\prod_{i=1}^{2{\ell}}\bigg(1-\frac{\alpha_i-\gamma}{x-\gamma}\bigg)^{1/2}
\]
converges away from $D$. Now, if $\alpha_1,\dots,\alpha_{\ell}\in \C_p$ are the roots of $h(x)$ in $\A^{1,\an}\setminus \overline{B}(0,R)$, then
\[
\bigg(\prod_{i=1}^{\ell}(x-\alpha_i)\bigg)^{1/2} =\prod_{i=1}^{\ell}\left((-\alpha_i)^{1/2}\left(1-\frac{x}{\alpha_i}\right)^{1/2}\right)
\]
converges on $B(0,R)$. By multiplying these functions, we get the desired square root of $h(x)$.
\end{proof}

\begin{prop} \label{p:eliminateroots} Let $U\subset \P^{1,\an}$ be a rational basic wide open that is good with respect to $S_f$, then $\pi^{-1}(U)$ is the union of at most two basic wide opens. 
\end{prop}

\begin{proof}
We will find a new coordinate $\tilde{y}$ and a polynomial $g(x)$ whose roots lie in distinct residue discs such that 
\begin{equation*} 
\pi^{-1}(U)\simeq\{(x,\tilde{y})\mid  x\in U, \tilde{y}^2=g(x)\}.
\end{equation*}

We view $U$ as a subset of $\P^{1,\an}$ where we have made a fractional linear transformation to ensure that the roots of $f(x)$ contained in $U$ are in distinct residue discs. We can suppose that either $U=\P^{1,\an}$ or that $U\subset \overline{B}(0,R)$ for some $R\in \|\bK^*\|_p$ with $R\geq 1$. If $U=\P^{1,\an}$, then $\pi^{-1}(U)=\tilde{X}$ is a good reduction curve.
Otherwise, write $D_1,\dots,D_m$ for the closed discs contained in $\overline{B}(0,R)$ (each contained in a distinct residue disc) in which $f(x)$ has multiple roots in their interior. 

We factor $f(x)=g(x)h(x)$ where 
\begin{enumerate}
    \item the polynomial $g(x)$ only has roots in $\overline{B}(0,R)$,
    \item the polynomial $g(x)$ has at most one root in each residue disc, and
    \item the polynomial $h(x)$ has an even number of roots in each $D_i$.
\end{enumerate}
Let $\ell(x)$ be a square root of $h(x)$ on $U$ which exists by Lemma~\ref{squarerootsofseries}, and set $\tilde{y}=\frac{y}{\ell(x)}$. Because $\ell(x)$ is non-vanishing on $U$, the map
\[(x,y)\mapsto (x,\tilde{y})\]
is invertible for $x$ in $U$ hence gives the desired isomorphism.

We now consider the complete curve $\tilde{\pi}\colon \tilde{X}\to \P^1$ given by compactifying $\tilde{y}^2=g(x)$. Because the roots of $g(x)$ lie in distinct residue discs, $\tilde{X}$ has a smooth model over $R$, the valuation ring of some field $\bK$. Since $\tilde{X}$ has good reduction, we need only show that $\pi^{-1}(U)$ is a basic wide open. If $g(x)$ is of degree $0$, $\tilde{X}$ is the union of two copies of $\P^1$.  In this case, $\pi^{-1}(U)$ is isomorphic to two copies of $U$. If $g(x)$ is of positive degree, we must identify $\pi^{-1}(U)$.

Write $U=\P^{1,\an}\setminus \big(\bigcup_{i=1}^m D_i\cup D_\infty\big)$ for closed discs $D_1,\dots,D_m$, each contained in a residue disc where $D_\infty$ is a disc of the form $\P^{1,\an}\setminus B(0,R')$ for some $R'>1$. We need to identify $\tilde{\pi}^{-1}(D_i)$. As we discussed, the polynomial $g(x)$ has at most one root in each $D_i$. 
Consider the case where $g(x)$ has no roots in $D_i$. Then $g(x)$ has an analytic square root on $D_i$, and $\tilde{\pi}^{-1}(D_i)$ is the union of two disjoint closed discs, each isomorphic to $D_i$.
Now, consider the case where $g(x)$ has exactly one root in $D_i$. By a fractional linear transformation, we may suppose that $g(x)=x$. Then the closed disc $D_i$ is of the form $\{x\mid \|x\|_p\leq r\}$ for some $r\in G$. Consequently,
\[\pi^{-1}(D_i)=\{(x,\tilde{y})\mid x\in D_i, \tilde{y}^2=x\}
=\{\tilde{y}\mid \|y\|_p\leq r^{1/2}\}\] 
is a closed disc. This lies in a residue disc in the model over $R$. A similar argument applies to $D_\infty$.
It follows that 
\[\pi^{-1}(U)=\tilde{X}\setminus \left(\bigcup_{i=1}^m \tilde{\pi}^{-1}(D_i)\cup \pi^{-1}(D_\infty)\right)\] is a basic wide open.
\end{proof}
 
Observe that in the above, $\tilde{X}$ has either one or two components according to where $g(x)$ is degree $0$ or not. We immediately see that $\pi^{-1}(U)$ is disconnected exactly when $f(x)$ has no roots in $U$ and has an even number of roots in each deleted disc $D_i$. In this case, we say that $U$ is \defi{even}. Otherwise, we say that it is \defi{odd}.

The double cover of an annulus (such as one arising as a component of the intersection of two elements of a semistable covering) given by $y^2=f(x)$ is the following.

\begin{lem}
Let $A$ be an annulus in $\A^{1,\an} \subset \P^{1,\an}$. Suppose that $S_f$ is disjoint from $A$.
Then $\pi^{-1}(A)=\{(x,y)\mid x\in A, y^2=f(x)\}$ is
\begin{enumerate}
\item \label{i:disjointannuli} the union of two disjoint annuli if $S_f$ has an even number of elements in each component of $\P^{1,\an}\setminus A$,
\item \label{i:singleannulus} an annulus if $S_f$ has an odd number of elements in each component of $\P^{1,\an}\setminus A$. 
\end{enumerate}
\end{lem}

\begin{proof}
By a fractional linear transformation, we may reduce to the case where $A=\{x\mid 1<\|x\|<r\}$ for some $r>1$. As in the proof of Proposition~\ref{p:eliminateroots}, we can reduce to the case where $\pi^{-1}(A)=\{(x,\tilde{y})\mid \tilde{y}^2=g(x), x\in A\}$ where $g(x)$ is of degree at most $1$. If $g(x)$ is of degree $0$, then we are in case \eqref{i:disjointannuli}. If $g(x)$ is of degree $1$, we can  reduce to the case where $g(x)=x$. Then, $\pi^{-1}(A)$ is given by $\{\tilde{y}\in\A^{1,\an}\mid 1<\|y\|<r^{1/2}\}$, and we are in case \eqref{i:singleannulus}. 
\end{proof}

We refer to an annulus $A$ as \defi{even} or \defi{odd} according to whether $\pi^{-1}(A)$ is disconnected or connected.

If $\cC$ is a semistable covering of $\P^{1,\an}$ that is good with respect to $S_f$, we can produce a semistable covering $\cD$ of $X^{\an}$ that is good with respect to the set of Weierstrass points $W$. We let $\cD$ be the set of components of $\pi^{-1}(U)$ as $U$ ranges over elements of $\cC$. In the case that $U\in\cC$ is even, $\pi^{-1}(U)$ will have two components which will give two elements of the covering of $X^{\an}$. Let $Y_1$ and $Y_2$ be two distinct elements of $\cD$; put $U_i = \pi(Y_i)$ for $i=1,2$. There are three possibilities for the intersection $Y_1\cap Y_2$; it is
\begin{enumerate}
    \item empty if $U_1=U_2$ or $U_1\cap U_2$ is empty,
    \item an annulus if $U_1\cap U_2$ is an odd annulus,
    \item the union of two disjoint annuli if  $U_1\cap U_2$ is an even annulus.
\end{enumerate}

Let $\Gamma(\cD,W)$ be the dual graph of the covering $\cD$ with respect to $W$. We will give a description of $\Gamma(\cD,W)$ in terms of $T$ similar to \cite[Section 6]{stoll2018uniform}. The dual graph will have both closed edges and half-open edges. Unless noted otherwise, edges are taken to be closed. We first designate half-open edges, edges and vertices of $T$ as even or odd. All half-open edges of $T$ are odd. An edge of $T$ is even exactly when the corresponding annulus is even. A vertex of $T$ is even exactly when all of its adjacent edges are even. For a vertex $v$, its \defi{genus} is the integer $g(v)$ satisfying
\[2g(v)-2=-4+n_o(v)\]
where $n_o(v)$ are the number of odd edges (including half-open edges) adjacent to $v$. Observe that even vertices have genus equal to $-1$ (corresponding to a disjoint union of two $\P^1$'s). By the Riemann--Hurwitz formula, $g(v)$ is the genus of the good reduction curve in which $\pi^{-1}(U)$ will be embedded.

\begin{defn} To $T$, we may attach a graph $\Gamma$. Let $\Gamma$ be the graph whose vertex set consists of
\begin{enumerate}
\item one vertex $\tilde{v}$ for each odd vertex $v$ of $T$; and
\item two vertices $\tilde{v}_+$, $\tilde{v}_-$ for each even vertex $v$ of $T$
\end{enumerate}
whose edge set is
\begin{enumerate}
\item one edge $\tilde{e}$ for each odd edge $e$ of $T$;
\item two edges $\tilde{e}_+,\tilde{e}_-$ for each even edge $e$ of $T$; and
\item one half-open edge $\tilde{e}$ for each half-open edge $e$ of $T$.
\end{enumerate}

For each adjacent pair $(v,e)$ of $T$ with $v$ and $e$ odd, we declare $\tilde{v}$ and $\tilde{e}$ adjacent. If $v$ is odd and $e$ is even, we declare $\tilde{v}$ and $\tilde{e}_\sigma$ adjacent for $\sigma=+,-$. If $v$ and $e$ are even, we declare $\tilde{v}_\sigma$ and $\tilde{e}_\sigma$ adjacent for $\sigma=+,-$. Because a half-open edge $e$ is only attached to an odd vertex $v$ of $T$, the corresponding half-open edge $\tilde{e}$ is attached to the vertex $\tilde{v}$ of $\Gamma$.

There is a natural map $\pi\colon \Gamma\to T$ taking $\tilde{v}$ or $\tilde{v}_+,\tilde{v}_-$ to $v$ and $\tilde{e}$ or $\tilde{e}_+, \tilde{e}_-$ to $e$. 
\end{defn}

\begin{eg}
As an illustration, consider the following tree $T$:
\[\begin{tikzpicture}
\draw [thick] (-4,0) -- (4,0);
\filldraw (-4,0) circle (2.5pt);
\filldraw (-2,0) circle (2.5pt);
\filldraw (0,0) circle (2.5pt);
\filldraw (2,0) circle (2.5pt);
\filldraw (4,0) circle (2.5pt);
\draw (-4,0) -- (-4.5,-0.7);
\draw (-4,0) -- (-4,-0.7);
\draw (-4,0) -- (-3.5,-0.7);
\draw (-2,0) -- (-2.0,-0.7);
\draw (2,0) -- (2,-0.7);
\draw (4,0) -- (4,-0.7);
\filldraw (-4,0.4) node{$v_1$};
\filldraw (-2,0.4) node{$v_2$};
\filldraw (0,0.4) node{$v_3$};
\filldraw (2,0.4) node{$v_4$};
\filldraw (4,0.4) node{$v_5$};
\filldraw (-3,0.25) node{$e_1$};
\filldraw (-1,0.25) node{$e_2$};
\filldraw (1,0.25) node{$e_3$};
\filldraw (3,0.25) node{$e_4$};
\end{tikzpicture}\]
The edges $e_2,e_3$ and the vertex $v_3$ are even; all the others are odd. Here is the corresponding graph $\Gamma$:
\[\begin{tikzpicture}
\draw [thick] (-4,3) -- (-2,3);
\draw [thick] (-2,3) -- (0,4);
\draw [thick] (-2,3) -- (0,2);
\draw [thick] (0,4) -- (2,3);
\draw [thick] (0,2) -- (2,3);
\draw [thick] (2,3) -- (4,3);
\filldraw (-4,3) circle (2.5pt);
\filldraw (-2,3) circle (2.5pt);
\filldraw (0,4) circle (2.5pt);
\filldraw (0,2) circle (2.5pt);
\filldraw (2,3) circle (2.5pt);
\filldraw (4,3) circle (2.5pt);
\draw (-4,3) -- (-4.5,2.3);
\draw (-4,3) -- (-4,2.3);
\draw (-4,3) -- (-3.5,2.3);
\draw (-2,3) -- (-2,2.3);
\draw (2,3) -- (2,2.3);
\draw (4,3) -- (4,2.3);
\filldraw (-4,3.4) node{$\tilde{v}_1$};
\filldraw (-2,3.4) node{$\tilde{v}_2$};
\filldraw (0,4.4) node{$\tilde{v}_{3,+}$};
\filldraw (0,1.6) node{$\tilde{v}_{3,-}$};
\filldraw (2,3.4) node{$\tilde{v}_4$};
\filldraw (4,3.4) node{$\tilde{v}_5$};
\filldraw (-3,3.3) node{$\tilde{e}_1$};
\node [thick,rotate=32] at (-1,3.8) {$\tilde{e}_{2,+}$};
\node [thick,rotate=336] at (-1,2.16) {$\tilde{e}_{2,-}$};
\node [thick,rotate=336] at (1.1,3.75) {$\tilde{e}_{3,+}$};
\node [thick,rotate=32] at (1.1,2.25) {$\tilde{e}_{3,-}$};
\filldraw (3,3.3) node{$\tilde{e}_4$};
\end{tikzpicture}\]
\end{eg}

By unwinding the description of $\Gamma$, we have the following proposition (from which one sees that $\Gamma$ is a graph structure on the skeleton of $X^{\an}\setminus W$ by reasoning identical to that of Remark~\ref{rmk:coveringtoskeleton}):
\begin{prop} 
The dual graph $\Gamma(\cD,W)$ is equal to $\Gamma$. 
\end{prop} 

Let $T_e$ be the union of even edges of $T$ and let $V_o$ be the odd vertices of $T$ that are adjacent to even edges. We will describe the first homology group of $\Gamma$ in terms of the relative homology group $H_1(T_e,V_o;\R)$ which is given as the kernel of the map
\[\partial\colon C_1(T_e;\R)\to C_0(T_e;\R)/C_0(V_o;\R).\]
Define a map $\iota\colon C_1(T_e;\R)\to C_1(\Gamma;\R)$ by $e\mapsto \tilde{e}_+-\tilde{e}_-$.

\begin{prop} \label{p:iota} The map $\iota$ induces an isomorphism $\iota\colon H_1(T_e,V_o;\R)\to H_1(\Gamma;\R)$.
\end{prop}

\begin{proof}
Define a map $\kappa\colon C_1(\Gamma;\R)\to C_1(T_e;\R)$
by 
\[\kappa(\tilde{e})=0,\ \kappa(\tilde{e}_+)=e,\  \kappa(\tilde{e}_-)=0.\]
We first show that $\kappa$ maps $H_1(\Gamma;\R)$ to $H_1(T_e,V_o;\R)$. Let $C\in H_1(\Gamma;\R)$. For an even vertex $v$, $\pi$ is an simplicial homeomorphism of the open star of $\tilde{v}_+$ onto its image. Because the coefficient of $\tilde{v}_+$ in $\partial C$ is zero, the coefficient of $v$ in $\partial(\kappa(C))$ is also zero. Consequently, we have $\kappa(C)\in H_1(T_e,V_o;\R)$.

Now, we claim that $\kappa$ and $\iota$, considered as maps between $H_1(\Gamma;\R)$ and $H_1(T_e,V_o;\R)$, are inverses of one another. Clearly $\kappa\circ\iota$ is the identity. We claim $\iota\circ\kappa$ is the identity. If $C=\sum_{e} a_e e$ is a cycle in $\Gamma$, then $\pi_*(C)=0$ in $H_1(T;\R)$. The only way that this can occur is if $a_{\tilde{e}}=0$ for all edges $\tilde{e}$ and $a_{\tilde{e}_+}=-a_{\tilde{e}_-}$ for all pairs $(\tilde{e}_+,\tilde{e}_-)$ above an even edge $e$. From this we can conclude that 
$C=\iota(\kappa(C))$.
\end{proof}

The following relation between the cycle pairing on $(T_e,V_o)$ and that on $\Gamma$ is straightforward: 

\begin{prop} \label{p:iotacycle}
Let $C,D\in H_1(T_e,V_o;\R)$. Then $\langle\iota(C),\iota(D)\rangle=2\langle C,D\rangle$.
\end{prop}

\section{Integrals on Hyperelliptic Basic Wide Opens} \label{IntegrationOnWideOpens}

\subsection{$1$-forms on Hyperelliptic Basic Wide Opens} Let $X$ be a hyperelliptic curve defined by $y^2 = f(x)$. In Section~\ref{HyperellipticCoverings}, we explained how to construct a semistable covering of $X^\an$ by hyperelliptic basic wide opens. In this section, we summarize Berkovich--Coleman integration algorithms on these spaces. We note that these are ordinary Coleman integrals; in particular, they are path-independent. 

We fix a covering as above and consider an element $Y$ of this covering. Let $\omega$ be an odd holomorphic $1$-form on $Y$. Recall that odd means that the hyperelliptic involution acts on $\omega$ as multiplication by $-1$. 

If the space $Y$ is isomorphic to the standard open disc (resp. a standard open annulus) with parameter $t$, then $\omega$ pulls back as $F(t)dt$ where $F(t)$ is a power (resp. Laurent) series. In this case one can compute the integral by antidifferentiating. 

For other spaces, in order to make  use of the existing explicit methods, we need to pass to a good reduction curve. By the proof of Proposition~\ref{p:eliminateroots}, the space $Y$ is isomorphic to a basic wide open space $Z$ inside the good reduction curve $\tilde{X}^{\an}$ given by $y^2=g(x)$ for some polynomial $g(x)$ of degree $d$. Note that if $d$ is odd, then $d=2g+1$, and if $d$ is even, $d=2g+2$ where $g$ is the genus of $\tilde{X}$. We will suppose that $g(x)\in\bK[x]$ for some finite extension $\bK$ of $\Q_p$. As we will discuss in Section~\ref{s:powerseriesexpansion}, the form $\omega$ pulls back to $Z$ as an odd $1$-form that can be expressed as a series of odd $1$-forms. By Proposition~\ref{swapintsum}, we can interchange the order of summation and integration. Thus we need to integrate terms in this series. Let $\eta$ denote such a term. Using the change-of-variables property for Berkovich--Coleman integrals, it suffices to compute the integral of $\eta$ on $Z$. On the other hand, we will see that the form $\eta$ extends to $\tilde{X}$ as a meromorphic form with poles outside of $Z$ and by Remark~\ref{BCintegralislocal} we can perform this integral on the complete curve $\tilde{X}$.

If we write $Z=\tilde{X}^{\an}\setminus \left(\bigcup_{i=1}^r D'_i\right)$ for closed discs  $D'_1,\dots,D'_r$ (which arise as preimages of discs closed in $\P^{1,\an}$), then by \cite[Propositions 4.3, 4.4]{coleman89:reciprocity_laws} (see also the discussion in \cite[Theorem 2.24]{krzb15:uniform_bounds}) the sequence
\[0 \rightarrow H^1_\dR(\tilde{X}) \rightarrow H^1_\dR(Z) \xrightarrow{\Dsum\Res} \Dsum_{i=1}^r \C_p \xrightarrow{\sum} \C_p \rightarrow 0\]
is exact where $\Res$ takes the residue around the $D'_i$'s and $\Sigma$ is summation. If we let the superscript ``$-$'' denote the $(-1)$-eigenspace of the maps induced by the hyperelliptic involution, we have the short exact sequence
\begin{equation}
\label{exactseqfirstdeRham}
0 \rightarrow H^1_\dR(\tilde{X})^{-} \rightarrow H^1_\dR(Z)^- \xrightarrow{\Dsum\Res}
\left(\Dsum_{i=1}^r \C_p\right)^-\rightarrow 0
\end{equation} 
which says that in order to obtain a spanning set for the odd part of the first de Rham cohomology of $Z$, we only need to adjoin $1$-forms with poles in the $D'_i$'s to a basis for the de Rham cohomology of $\tilde{X}$. Here the hyperelliptic involution acting on the last factor exchanges the residues around hyperelliptically conjugate discs and acts as the identity on residues around discs containing a Weierstrass point.

We now consider the case where $d=\text{deg}(g(x))\geq 3$ in which case the curve $\tilde{X}$ is hyperelliptic. 
Extend the field $\bK$ so that it contains the roots of $g(x)$. 

By our construction, the $D'_i$'s arise as components of the preimages under $\pi\colon \tilde{X}^{\an}\to \P^{1,\an}$ of some closed discs $D_1,\dots,D_n$ in $\A^{1,\an}$, each contained in a distinct residue disc and possibly also of a disc $D_\infty$ around $\infty$. Such a disc is called \defi{Weierstrass} if it contains a root of $g(x)$. Suppose that we have ordered the discs such that $D_1,\dots,D_k$ are the non-Weierstrass discs and $D_{k+1},\dots,D_n$ are the Weierstrass discs. Observe that $D_\infty$ is Weierstrass if and only if $g(x)$ is of odd degree.

Let $\beta_1,\dots,\beta_n$ be elements of $\P^1(\bK)$ contained in $D_1,\dots,D_n$. We choose $\beta_i$ to be a root of $g(x)$ if $D_i$ is a Weierstrass disc contained in $\A^{1,\an}$. For $D_\infty$, choose $\beta_\infty=\infty$. Define the forms
\[\left\{\nu_j=\frac{dx}{(x-\beta_j)2y}\right\}_{j=1,\dots,k}\] 
where the form $\nu_j$ has simple poles at the hyperelliptically conjugate points $\pi^{-1}(\beta_j)$.

For an integer $i$, define the $1$-form \[\omega_i=x^i\frac{dx}{2y}.\]
In both the odd and even degree cases, $\{\omega_0,\dots,\omega_{d-2},\nu_1,\dots,\nu_k\}$ will form a spanning set for $H^1_\dR(Z)^-$. Consequently,
\[\eta=dF+\sum_{i=0}^{d-2} c_i\omega_i + \sum_{j=1}^k d_j\nu_j\]  
holds for an analytic function $F$ on $Z$ and $c_i,d_j \in\bK$. For points $R,S\in Z(\C_p)$, the equality above gives
\[\BCint_S^R\eta=F(R)-F(S)+\sum_{i=0}^{d-2} c_i\BCint_S^R\omega_i + \sum_{j=1}^k d_j\BCint_S^R\nu_j.\] 
Below, we will explain how to compute the integrals on the right.

\subsection{Summary of Integration Algorithms}
We will first state the algorithms when $g(x)$ is of odd degree where they are most fully developed. There is partial work in the even degree case, and one can apply a fractional linear transformation to $\P^1$ to transform the even degree case to the odd degree case.

We start with the integrals $\lBCint_S^R\omega_i$. The paper \cite{BBKExplicit} describes a method for computing Coleman integrals of those meromorphic forms whose poles all belong to Weierstrass residue discs.

If the points $R$ and $S$ lie in the same residue disc, in which case we refer to the integral as a \emph{tiny integral}, we may use the following lemma.

\begin{lem} \label{tinyintegral} 
\emph{(\cite[Algorithm 8]{BBKExplicit})}
For points $R,S\in \tilde{X}(\C_p)$ in the same residue disc, neither equal to the point at infinity, we have
\[\BCint_S^R \omega_i = \int_0^1 \frac{x(t)^i}{2y(t)}\frac{dx(t)}{dt}dt\]
where $(x(t),y(t))$ is a linear interpolation from $S$ to $R$ in terms of a local coordinate $t$. We can similarly integrate any form that is holomorphic in the residue disc containing the endpoints.  
\end{lem} 

If the points $R$ and $S$ lie in distinct non-Weierstrass residue discs, the method of tiny integrals is not available. Coleman's idea was to extend the notion of integration by \emph{analytic continuation along Frobenius}. Let $\phi$ be the lift of Frobenius constructed in \cite[Algorithm 10]{BBKExplicit}. This map is rigid analytic; moreover it maps a $\Q_p$-point into its residue disc. By the change-of-variables formula with respect to $\phi$, we have the following theorem.

\begin{thm} \label{mainalginBBK}
Let $M$ denote the matrix over $\bK$ such that
\begin{equation} \label{kedlayaalgortihm}
\phi^*\omega_i = df_i+\sum_{j=0}^{2g-1}M_{ij}\omega_j
\end{equation} 
for all $i=0,1,\dots,2g-1$. Then, for points $R,S\in \tilde{X}(\Q_p)$ in distinct non-Weierstrass residue discs, we have the equality
\[\sum_{j=0}^{2g-1}(M-I)_{ij}\BCint_S^R\omega_j=f_i(S)-f_i(R)-\BCint_S^{\phi(S)}\omega_i-\BCint_{\phi(R)}^R\omega_i.\]
Moreover, the matrix $M-
I$ is invertible (see \cite[Section 2]{kedlaya2001counting}), and we can solve this linear system to obtain the integrals $\lBCint_S^R\omega_i$.
\end{thm}

\begin{proof}
The terms in equation \eqref{kedlayaalgortihm} can be calculated using Kedlaya's algorithm \cite[Algorithm 10]{BBKExplicit}. The result follows from Algorithm 11 and Remark 13 in \cite{BBKExplicit}.
\end{proof}

Thanks to this theorem, beyond evaluating primitives, computing tiny integrals and solving a linear system, the matrix of Frobenius is the only data that is needed to compute Coleman integrals between endpoints in distinct non-Weierstrass residue discs.

Suppose now that $R'$ and $S'$ are points, at least one of which is Weierstrass, lying in different residue discs. The following lemma will be useful.

\begin{lem} \label{weierstrassendpoints}
\emph{(\cite[Lemma 16]{BBKExplicit})} Let $\omega$ be an odd meromorphic $1$-form on $\tilde{X}$. For points $R',S'\in \tilde{X}(\C_p)$ which are not poles of $\omega$, such that $S'$ is a Weierstrass point, we have
\[\BCint_{S'}^{R'}\omega = \frac{1}{2}\BCint_{w(R')}^{R'}\omega.\]
In particular, if $R'$ is also a Weierstrass point, then $\lBCint_{S'}^{R'}=0$.
\end{lem}

\begin{proof} This follows from $\lBCint_{S'}^{R'}\omega = \lBCint_{S'}^{w(R')}(-\omega) = \lBCint_{w(R')}^{S'}\omega$ and additivity in endpoints.
\end{proof}

If $S$ lies in a finite Weierstrass residue disc containing Weierstrass point $S'$, Lemma~\ref{weierstrassendpoints} gives \[\BCint_{S}^{R}\omega_i = \BCint_{S}^{S'}\omega_i + \frac{1}{2}\BCint_{w(R)}^{R}\omega_i.\] 
If $R$ does not belong to a Weierstrass residue disc, the second integral can be calculated using Theorem~\ref{mainalginBBK}; if $R$ also lies in a finite Weierstrass residue disc containing Weierstrass point $R'$, then by Lemma~\ref{weierstrassendpoints} again, we have
\[\BCint_{S}^{R}\omega_i = \BCint_{S}^{S'}\omega_i + \BCint_{R'}^{R}\omega_i.\] 
These tiny integrals can be computed using Lemma~\ref{tinyintegral}.

Now, we consider the integrals $\lBCint_S^R\nu_j$. As we discussed before, the form $\nu_j$ has poles at the hyperelliptically conjugate points $\pi^{-1}(\beta_j)$. The above approach does not work for this case, however, the paper \cite{BBComputing} provides a new method. 

First, consider the case where $R$ and $S$ lie in the same residue disc. If the form $\nu_j$ is holomorphic in the disc, then we can compute its integral as in Lemma~\ref{tinyintegral}. Otherwise, we make use of the following lemma in which we decompose our form in the disc into the sum of a holomorphic form and a logarithmic differential (i.e. a differrential of the form $df/f$ for $f\in \bK(X)^{\times}$).

\begin{lem} \label{tinywithpoles}
\emph{(\cite[Lemma 4.2]{balakrishnan2016quadratic})} Let $P$ be a non-Weierstrass point and set 
\[\nu = \frac{y(P)}{x-x(P)}\frac{dx}{y}.\]
For points $R,S$ different from $P$ but contained in the residue disc of $P$, we have
\[\BCint_S^R \nu = \BCint_S^R \frac{g(x(P))-g(x)}{y(x-x(P))(y(P)+y)}dx + \operatorname{Log}\bigg(\frac{x(R)-x(P)}{x(S)-x(P)}\bigg)\] 
where the integrand on the right side is holomorphic on the residue disc.
\end{lem}

Now, we examine the case where  $R$ and $S$ lie in distinct residue discs. As before, using Lemma~\ref{weierstrassendpoints}, we may reduce to the case that the residue discs are non-Weierstrass. Before stating the theorem to deal with this case, we recall the following objects from \cite{coleman1989}:

\begin{itemize}
    \item The space $H_{\dR}^1(X/\bK)$ has a canonical non-degenerate alternating form given by the algebraic cup product 
    \[\cup:H_{\dR}^1(X/\bK)\times H_{\dR}^1(X/\bK)\to \bK.\] 
    This pairing may be described using a well-known formula of Serre:
    \[[\mu_1]\cup[\mu_2]=\sum_{P\in X(\C_p)} \operatorname{Res}_P \left(\mu_2\int\mu_1\right).\]
    \item Let $T(\bK)$ denote the subgroup of differentials of the third kind. We denote the subgroup consisting of the logarithmic differentials by $T_l(\bK)$. By \cite[Proposition 2.5]{coleman1989}, there is a canonical homomorphism 
    \[\Psi:T(\bK)/T_l(\bK)\to H_{\dR}^1(X/\bK)\]
    which is the identity on differentials of the first kind. The map $\Psi$ can be extended to a linear map from the $\bK$-vector space of all meromorphic differentials on $X/\bK$ to $H_{\dR}^1(X/\bK)$ as follows. First, we express a given meromorphic differential $\eta$ as $\eta = \sum a_i\nu_i + \mu$, where $\nu_i$'s are of the third kind, $a_i\in \bar{\bK}$ and $\mu$ is of the second kind; then we define $\Psi(\eta) = \sum a_i\Psi(\nu_i) + [\mu].$
\end{itemize}

We have the following which we state for curves and points defined over $\Q_p$: 
\begin{thm} \label{mainalginBB12}
Suppose the curve $\tilde{X}$ is defined over $\Q_p$ and the polynomial $g(x)$ is monic. Let $P$ and $\nu$ be as in Lemma~\ref{tinywithpoles}. For points $R,S\in Z(\Q_p)$ in distinct non-Weierstrass residue discs, not equal to $P$ and $w(P)$, we have
\[\BCint_S^R \nu = \frac{1}{1-p}\bigg(\Psi(\alpha)\cup\Psi(\beta)+\sum_{A\in \tilde{X}(\C_p)}\operatorname{Res}_A\bigg(\alpha\int\beta\bigg)-\BCint_{\phi(S)}^S\nu-\BCint_R^{\phi(R)}\nu\bigg)\]
where $\alpha = \phi^*\nu-p\nu$, $\beta$ is a form with $\operatorname{Res}(\beta) = R-S$.
\end{thm}

\begin{proof}
See Algorithm $4.8$ and Remark $4.9$ in \cite{BBComputing}.
\end{proof}

\begin{rem}\label{evendegreecase}
The generalization of Theorem~\ref{mainalginBB12} to even degree case will be discussed in \cite{gajovic}. Combining this with the techniques in \cite{balakrishnan2015coleman}, which extend the algorithms in \cite{BBKExplicit} to even degree models, one should be able to do the computation above for even degree case.
\end{rem}

Now, we consider the case where $g(x)$ is of degree at most $2$. The curve $\tilde{X}$ is rational and therefore $H^1_\dR(\tilde{X})$ is trivial. By the exact sequence \eqref{exactseqfirstdeRham}, our form $\eta$ will be sum of an exact form $dF$ and forms with simple poles. Moreover, using the equation $y^2=g(x)$, one can easily express the non-exact part as a sum of logarithmic differentials $c_i dF_i/F_i$ for constants $c_i$. This gives, 
\[\BCint \eta = F + \sum_{i} c_i \operatorname{Log}(F_i).\]

\section{Decomposition of $1$-forms with Specified Poles}\label{Decompositionof1forms}

We will now consider $1$-forms with poles in a specified set. Let $X$ be the good reduction hyperelliptic curve defined by $y^2=g(x)$ where $g(x)$ is of degree $d$. Moreover, we assume that the polynomial $g(x)$ is monic with integral coefficients in some finite extension of $\Q_p$; this assumption guarantees that the $p$-adic absolute value of roots of $g(x)$ are at most $1$. Let $Y$ be a basic wide open in $X$ contained in $B(0,R)$ for some $R\in \|\bK^*\|_p$ with $R>1$

Let  $T=\{\beta_1,\dots,\beta_\ell\}$ be a subset of $\A^1(\bK)$ for some finite extension $\bK$ of $\Q_p$. We will study $1$-forms of the form
\[\eta=x^{n_{\infty}}\prod_{j=1}^\ell \frac{1}{(x-\beta_{j})^{n_j}}\frac{dx}{2y}\]
for nonnegative integers $n_1,\dots,n_\ell,n_{\infty}$. We will further suppose that $\|\beta_i\|_p\leq 1$ for all $i$. Below, we will make 
the following assumptions: for $i=1,\dots,k$, we have $\|g(\beta_i)\|_p=1$; and for $i=k+1,\dots,\ell$, we have $g(\beta_i)=0$ and 
$\|g'(\beta_i)\|_p=1$. 

We will soon need to consider a series of $1$-forms whose terms are of the above form. To integrate them, we will interchange integration and summation using Proposition~\ref{swapintsum}. We will provide an algorithm to express the $1$-forms in terms of our given basis: the $1$-form $\eta$ can be written as 
\begin{equation} \label{e:decompofeta}
\eta=dF+\sum_{i=0}^{d-2} c_i\omega_i + \sum_{j=1}^k d_j\nu_j
\end{equation}
where $F$ is an analytic function on $Y$. Furthermore, we will find bounds on $c_i$, on $d_j$, and on the maximum value of the norm of $F$ on $Y$ in Proposition~\ref{p:primitive}. 

\begin{rem} \label{exactdifferentials}
We will make use of two types of exact $1$-forms.
\begin{enumerate}
    \item For a positive integer $m$, consider 
    \[d\left(\frac{y}{(x-\beta)^m}\right) = \frac{(x-\beta)g'(x)-2mg(x)}{(x-\beta)^{m+1}}\frac{dx}{2y}.\]
    Such a form has poles at the points above $\beta$ and possibly also at the point(s) at infinity. If $\beta$ is a root of $g(x)$, the pole is of order $2m$ at $\pi^{-1}(\beta)$; in fact, if we write $\Symb$ for the monomial involving the highest order power of $(x-\beta)^{-1}$, we have
    \[\Symb\left(\frac{(x-\beta)g'(x)-2mg(x)}{(x-\beta)^{m+1}}\frac{dx}{2y}\right)=\Symb\left(\frac{(1-2m)g'(\beta)}{(x-\beta)^m}\frac{dx}{2y}\right).\]
    If $\beta$ is not a root of $g(x)$, there are poles of order $m+1$ at each of the points of $\pi^{-1}(\beta)$; in fact, we have
    \[\Symb\left(\frac{(x-\beta)g'(x)-2mg(x)}{(x-\beta)^{m+1}}\frac{dx}{2y}\right)=\Symb\left(\frac{-2mg(\beta)}{(x-\beta)^{m+1}}\frac{dx}{2y}\right).\]
    The two cases differ because $y$ is a uniformizer in one case, while $x-\beta$ is a uniformizer in the other.
    \item For a nonnegative integer $m$, consider
    \[d(x^my) = \big(x^mg'(x)+2 mx^{m-1}g(x)\big)\frac{dx}{2y}.\]
    Such a form has poles at the point(s) at infinity. Notice that the leading coefficient of $x^mg'(x)+2 mx^{m-1}g(x)$
    is $d+2m$ as the polynomial $g(x)$ is monic.
\end{enumerate}
\end{rem}

\subsection{Principal Parts}
We will write our $1$-form as in \eqref{e:decompofeta} by subtracting off the exact $1$-forms in Remark~\ref{exactdifferentials} to cancel the non-simple poles. To do so, we use the language of principal parts. 

\begin{defn} Let $\alpha$ be a smooth point of a curve $X$ and pick a uniformizer $t$ on $X$ for $\alpha$. 
For a meromorphic function $h$, the \defi{principal part of $h$} near $\alpha$ is the polynomial in $t^{-1}$ given by the negative degree terms in the Laurent expansion of $h$ in $t$. Let $\omega$ be a meromorphic $1$-form on $X$ that is regular and non-vanishing at $\alpha$. For a meromorphic $1$-form $\eta$ on $X$, $\frac{\eta}{\omega}$ is a meromorphic function defined in a punctured neighborhood of $\alpha$. The \defi{principal part} $\PP_{\omega,\alpha}(\eta)$ of $\eta$ near $\alpha$ with respect to $\omega$ is the principal part of $\frac{\eta}{\omega}$ near $\alpha$.
\end{defn}

For $\beta\in \A^1(\bK)$ and $\alpha\in\pi^{-1}(\beta)$, we have convenient choices for coordinates and $1$-forms. The $1$-form $\omega_0=\frac{dx}{2y}$ is regular and non-vanishing away the point(s) at infinity. Let $\eta$ be an odd $1$-form; then $\frac{\eta}{\omega_0}$ is invariant under the hyperelliptic involution. 

We first explain how to pick a uniformizer at Weierstrass points.  Let $\beta$ be a root of $g(x)$, then $y$ is a uniformizer at $\alpha=\pi^{-1}(\beta)$. However, we can pick a slightly more convenient uniformizer. We know that $\frac{x-\beta}{g(x)}$ does not vanish in a neighborhood of $\beta$ and so has an analytic square root $h(x)$ there. Then $w=yh(x)$ is a uniformizer at $\alpha$. Because the meromorphic function $\frac{\eta}{\omega_0}$ is invariant under the hyperelliptic involution, it can be written near $\alpha$ as a Laurent series in $w^2=x-\beta$. Therefore, the principal part of $\frac{\eta}{\omega_0}$ is a polynomial in $z=(yh(x))^{-2}=(x-\beta)^{-1}$. 

If $\beta$ is not a root of $g(x)$, then $\pi^{-1}(\beta)=\{\alpha_1,\alpha_2\}$ and $x-\beta$ is a uniformizer near both $\alpha_1$ and $\alpha_2$. Consequently, $\PP_{\omega_0,\alpha_1}(\eta)=\PP_{\omega_0,\alpha_2}(\eta)$. In this case, the principal part of $\frac{\eta}{\omega_0}$ is a polynomial in $z=(x-\beta)^{-1}$.

In any case, by using the Taylor expansion for $g(x)$ at $\beta$, we compute
\[\PP_{\omega_0,\alpha}\left(d\left(\frac{y}{(x-\beta)^m}\right)\right)
=\sum_{k=0}^{m}\frac{-m-k}{(m-k)!}g^{(m-k)}(\beta) z^{k+1}.\]
By a straightforward argument obtained by writing $x^k=(x-\beta+\beta)^k$ and using the integrality of binomial coefficients, one sees that the $p$-adic absolute value of the coefficients of the principal part are bounded above by $\max(1,\|\beta\|_p^{d})$. In particular, if $\|\beta\|_p\leq 1$, we have
\begin{equation} \label{e:ppexactbound}
\left\|\PP_{\omega_0,\alpha}\left(d\left(\frac{y}{(x-\beta)^m}\right)\right)\right\|_p\leq 1.
\end{equation} Here, for a polynomial $q(t)$, we define the value $\|q(t)\|_p$ as the maximum of the $p$-adic absolute value of its coefficients.

Recall that, for an integer $i$, $\omega_i$ is defined as $x^i\frac{dx}{2y}$. If $\eta$ is an odd $1$-form, $\frac{\eta}{\omega_i}$ is a meromorphic function on $\P^1$, and we may speak of its pole order at $\infty\in\P^1(\bK)$. In analogy with the above, we choose $\frac{1}{x}$ as a uniformizer at $\infty$. Write $\PP_{\omega_i,\infty}(\eta)$ for the principal part of $\frac{\eta}{\omega_i}$  considered as a meromorphic function on $\P^1$. Observe that for $\eta$ regular on the finite part of $Y$, if 
\[\frac{\eta}{\omega_0}=a_0+a_1x+\dots+a_{d-2}x^{d-2}\]
then
\[\PP_{\omega_{-1},\infty}(\eta)=\frac{\eta}{\omega_{-1}}=a_0\left(\frac{1}{x}\right)^{-1}+a_1\left(\frac{1}{x}\right)^{-2}+\dots+a_{d-2}\left(\frac{1}{x}\right)^{-(d-2)}\]
and
\[\eta=a_0\omega_0+a_1\omega_1+\dots+a_{d-2}\omega_{d-2}.\]
We will abuse notation and refer to the degree of the polynomial $\frac{\eta}{\omega_0}$ as the pole order at $\infty$. If the pole order is at most $d-2$, the above formula lets us determine the cohomology class of $\eta$. 
\subsection{Pole Reduction}
Using ideas similar to those of Tuitman \cite{Tuitman:counting1,Tuitman:counting2}, we will subtract off exact $1$-forms to lower the pole orders of $\eta$ at the $\beta$'s. We begin by cancelling the poles of $\eta$ of order greater than $1$ at non-Weierstrass points and the poles of $\eta$ at Weierstrass points. Then, we will cancel the simple poles at non-Weierstrass points by subtracting off multiples of $\nu_j$. The remainder $\eta'$ can be expressed in terms of the $\omega_i$'s by examining $\PP_{\omega_{-1},\infty}(\eta')$.

Define meromorphic $1$-forms $\mu_{\beta,m}$ by 
\[\mu_{\beta,m}=d\left(\frac{y}{(x-\beta)^m}\right).\]
We omit the proof of the following lemma (which is a computation in coordinates). 

\begin{lem} \label{l:ppatinfty} We have the following:
\begin{enumerate}
\item the pole order of $\infty$ of $\mu_{\beta,m}/\omega_{-1}$ is at most $d+1-m$,
\item the principal part of $\mu_{\beta,m}$ at $\infty$ obeys $\left\|\PP_{\omega_{-1},\infty}(\mu_{\beta,m})\right\|_p\leq 1$, and
\item the principal part of $d(x^my)$ at $\infty$ obeys $\|\PP_{\omega_{-1},\infty}(d(x^my))\|_p\leq 1$.
\end{enumerate}
\end{lem}

Below, we will make use of Legendre's formula for the $p$-adic valuation of factorials for $p\neq 2$. We have the bounds
\[\frac{m}{p-1}-\lceil\log_p(m)\rceil\leq v_p(m!)\leq \frac{m}{p-1}.\]
From this, we can obtain the following bound on odd factorials:
\[v_p((2m-1)!!)\leq \frac{m}{p-1}+\lceil\log_p(m)\rceil\leq \frac{m}{p-1}+\log_p(m)+1.\]

\begin{lem} \label{polereductionnonWss}
Let $\beta\in\A^1(\bK)$ with $\|\beta\|_p\leq 1$ and $\|g(\beta)\|_p=1$. Take $\alpha\in\pi^{-1}(\beta)$ and set $z=(x-\beta)^{-1}$. Let $\eta$ be an odd meromorphic $1$-form on $X$ such that
\[m\coloneq\deg_z(\PP_{\omega_0,\alpha}(\eta))-1>0.\]
Then there exists a unique polynomial $q(t)\in \bK[t]$ of degree $m$ such that 
\[\eta'\coloneq\eta-d(q((x-\beta)^{-1})y)\]
has at worst simple poles at points above $\beta$
and 
\[\|q(t)\|_p\leq p^{m/(p-1)}\|\PP_{\omega_0,\alpha}(\eta)\|_p.\]
Moreover, $\PP_{\omega_{-1},\infty}(\eta-\eta')$ is a $\bK$-linear combination of $\PP_{\omega_{-1},\infty}(\mu_{\beta,1}),\dots, \PP_{\omega_{-1},\infty}(\mu_{\beta,d-1})$
with coefficients with norm at most 
\[p^{m/(p-1)}\|\PP_{\omega_0,\alpha}(\eta)\|_p.\]
\end{lem}

\begin{proof}
Let $V$ be the $m$-dimensional $\bK$-vector space spanned by the meromorphic functions 
\[\{(x-\beta)^{-1}y,(x-\beta)^{-2}y,\dots,(x-\beta)^{-m}y\}\]
and let $W$ be the $\bK$-vector space spanned by $\{z^2,z^3,\dots,z^{m+1}\}$. Define
\begin{eqnarray*}
L\colon V&\to& W\\
h&\mapsto&T(\PP_{\omega_0,\alpha}(dh))
\end{eqnarray*}
where $T$ takes $z^1\mapsto 0$ and $z^i\mapsto z^i$ for $i\geq 2$. By Remark~\ref{exactdifferentials}, its matrix $M$ in these bases is upper triangular. In fact, the diagonal entries of $M$ are
\[M_{ii}=(-2i)g(\beta).\] 
As $g(\beta)\neq 0$, the matrix $M$ is invertible and we can find a polynomial $q(t)$ such that $q((x-\beta)^{-1})y\in V$ satisfies $L(q((x-\beta)^{-1})y)=T(\PP_{\omega_0,\alpha}(\eta))$. 

To get control over the coefficients of $q(t)$, we will use Cramer's rule. The coefficients of $q(t)$ are equal to $\det(M_{j})/\det(M)$ where $M_{j}$ is the matrix formed by replacing the $j$th column of $M$ by the coefficients of $T(\PP_{\omega_0,\alpha}(\eta))$. By Legendre's formula, we have
\[\|\det(M)\|_p\geq p^{-m/(p-1)}.\] 
By \eqref{e:ppexactbound}, the coefficients of $M$ are bounded above in $p$-adic absolute value by $1$, 
so 
\[\|\det(M_{j})\|_p\leq \|T(\PP_{\omega_0,\alpha}(\eta))\|_p.\]
Consequently, the coefficients of $q(t)$ are bounded above by
\[p^{m/(p-1)}\|\PP_{\omega_0,\alpha}(\eta)\|_p.\]
The bound on pole order at infinity and on the coefficients of the principal part at $\infty$ follow from Lemma~\ref{l:ppatinfty}.
\end{proof}

\begin{algorithm}[H] 
\label{PoleRednonWeierstrass}
\caption{\bf Pole reduction at finite non-Weierstrass points}

\KwIn{
\begin{itemize}
    \item $\alpha\in\pi^{-1}(\beta)$ where $\beta$ is a non-root of $g(x)$.
    \item An odd meromorphic form $\eta$ with pole at $\alpha$ of order $m$. 
\end{itemize}
}
\KwOut{A function $F$ such that the form $\eta-dF$ has at worst simple poles at points above $\beta$.}

\begin{enumerate}
    \item For $j=1,\dots,m-1$, compute the expansions \[\mu_{\beta,j} = \left(w_{-j-1}(x-\beta)^{-j-1} + \text{higher order terms}\right)
    \frac{dx}{2y}.\]
    \item Until $\eta$ has at worst a simple pole at $\alpha$, do the following: 
    \begin{enumerate}
        \item Compute the expansion
        \[\eta = \left(u_{-j}(x-\beta)^{-j} + \text{higher order terms}\right) \frac{dx}{2y}.\]
        \item Define $a_{j-1} = u_{-j}/w_{-j}$ and set $\eta \coloneq \eta - a_{j-1}\mu_{\beta,j-1}$.
    \end{enumerate}
    \item Return 
    \[F = \bigg(\frac{a_{m-1}}{(x-\beta)^{m-1}}+\frac{a_{m-2}}{(x-\beta)^{m-2}}+\dots+\frac{a_1}{x-\beta}\bigg)y. \]
\end{enumerate}
\end{algorithm}

\smallskip

Now, we consider a root $\beta$ of $g(x)$.

\begin{lem} 
\label{polereductionWss}
Let $\beta\in\A^1(\bK)$ be a root of $g(x)$ so that $\|\beta\|_p\leq 1$. Suppose $\|g'(\beta)\|_p=1$. Let $\alpha=\pi^{-1}(\beta)$ and set $z=(x-\beta)^{-1}$. Let $\eta$ be an odd meromoprhic $1$-form on $X$ such that
\[m\coloneq\deg_z(\PP_{\omega_0,\alpha}(\eta))>0.\]
Then there exists a unique polynomial $q(t)\in \bK[t]$ of degree $m$ such that 
\[\eta'\coloneq\eta-d(q((x-\beta)^{-1})y)\]
is regular at $\alpha$
and 
\[\|q(t)\|_p\leq mp^{1+m/(p-1)}\|\PP_{\omega_0,\alpha}(\eta)\|_p.\]
Moreover, $\PP_{\omega_{-1},\infty}(\eta-\eta')$ is a $\bK$-linear combination of $\PP_{\omega_{-1},\infty}(\mu_{\beta,1}),\dots, \PP_{\omega_{-1},\infty}(\mu_{\beta,d-1})$
with coefficients with norm at most 
\[mp^{1+m/(p-1)}\|\PP_{\omega_0,\alpha}(\eta)\|_p.\]
\end{lem}

\begin{proof}
Let $V$ be the $m$-dimensional $\bK$-vector space spanned by the meromorphic functions 
\[\{(x-\beta)^{-1}y,(x-\beta)^{-2}y,\dots,(x-\beta)^{-m}y\}\]
and let $W$ be the $\bK$-vector space spanned by $\{z,z^2,z^3,\dots,z^m\}$.
Define the map
\begin{eqnarray*}
L\colon V&\to& W\\
h&\mapsto&\PP_{\omega_0,\alpha}(dh).
\end{eqnarray*}
By Remark~\ref{exactdifferentials}, its matrix $M$ (in these bases) is an upper triangular matrix with diagonal entries
\[M_{ii}=(1-2i)g'(\beta).\]
Because $\|g'(\beta)\|_p=1$, $M$ is nonsingular and we can find $q((x-\beta)^{-1})y\in V$ with $L(q((x-\beta)^{-1})y)=\PP_{\omega_0,\alpha}(\eta)$. 

Again, we will use Cramer's rule to get control over the coefficients
of $h$. The determinant of $M$ has $p$-adic absolute value
\[\|\det(M)\|_p=\|(2m-1)!!\|_p\geq \frac{1}{m}p^{-1-m/(p-1)}\] where the last inequality follows from Legendre's formula for odd factorials. Let $M_{j}$ be the matrix formed by replacing the $j$th column of $M$ by the coefficients of $\PP_{\omega_0,\alpha}(\eta)$. By \eqref{e:ppexactbound}, the coefficients of $M$ are bounded above in $p$-adic absolute value by $1$, 
so 
\[\|\det(M_{j})\|_p\leq \|\PP_{\omega_0,\alpha}(\eta)\|_p.\]
Consequently, the coefficients of $p(t)$ are bounded above by
\[mp^{1+m/(p-1)}\|\PP_{\omega_0,\alpha}(\eta)\|_p.\]
The bound on pole order at infinity and on the coefficients of the principal part at $\infty$ again follow from Lemma~\ref{l:ppatinfty}.
\end{proof}

\begin{algorithm}[H] 
\label{PoleRedWeierstrass}
\caption{\bf Pole reduction at finite Weierstrass points}

\KwIn{
\begin{itemize}
    \item $\alpha=\pi^{-1}(\beta)$ where $\beta$ is a root of $g(x)$.
    \item An odd meromorphic form $\eta$ with pole at $\alpha$ of order $2m$. 
\end{itemize}
}
\KwOut{A function $F$ such that the form $\eta-dF$ is regular at $\alpha$.}

\begin{enumerate}
    \item For $j=1,\dots,m$, compute the expansions \[\mu_{\beta,j} = \left( w_{-j}(x-\beta)^{-j} + \text{higher order terms}\right)\frac{dx}{2y}.\]
    \item Until $\eta$ is regular at $\alpha$, do the following:
    \begin{enumerate}
        \item Compute the expansion 
        \[\eta = \left(u_{-j}(x-\beta)^{-j} + \text{higher order terms}\right)\frac{dx}{2y}.\]
        \item Define $a_{j} = u_{-j}/w_{-j}$ and set $\eta \coloneq \eta - a_{j}\mu_{\beta,j}$.    \end{enumerate}
    \item Return \[F=\bigg(\frac{a_m}{(x-\beta)^{m}}+\frac{a_{m-1}}{(x-\beta)^{m-1}}+\dots+\frac{a_1}{x-\beta}\bigg)y.\]
\end{enumerate}
\end{algorithm}

\smallskip

The main difference between Algorithms~\ref{PoleRednonWeierstrass} and \ref{PoleRedWeierstrass} is that, by subtracting off exact forms, poles at Weierstrass points can be removed completely but only non-simple poles can be removed at non-Weierstrass points. 

We will also need to lower the power of $x$ in the numerator of a $1$-form. This is the order reduction step in Kedlaya's algorithm.

\begin{lem} Let $\eta$ be an odd meromorphic $1$-form on $X$ such that
\[m\coloneq\deg_x(\PP_{\omega_{-1},\infty}(\eta))-(d-1)>0.\]
Let $T$ be the truncation of a polynomial to degree $d-1$ and $U=\operatorname{Id}-T$.
Then there exists a unique polynomial $q(t)\in \bK[t]$ of degree $m$ such that 
\[\eta'\coloneq\eta-d(q(x)y)\]
has $\deg_x(\PP_{\omega_{-1},\infty}(\eta'))\leq d-1$ and
\[\|q(t)\|_p\leq d(d+m)p^{2+m/(p-1)}\|U(\PP_{\omega_{-1},\infty}(\eta))\|_p.\]
Moreover,
$T(\PP_{\omega_{-1},\infty}(\eta-\eta'))$ has coefficients with norm at most
\[d(d+m)p^{2+m/(p-1)}\|U(\PP_{\omega_{-1},\infty}(\eta))\|_p.\]
\end{lem} 

\begin{proof}
Let $V$ be the $m$-dimensional $\bK$-vector space spanned by the meromorphic functions 
\[\{y,xy,\dots,x^{m-1}y\}\]
and let $W$ be the $\bK$-vector space spanned by $\{x^{d},x^{d+1},\dots,x^{d+m-1}\}$.
Define the map
\begin{eqnarray*}
L\colon V&\to& W\\
h&\mapsto&U(\PP_{\omega_{-1},\infty}(dh)).
\end{eqnarray*}
By Remark~\ref{exactdifferentials}, its matrix $M$ in these bases is upper triangular with diagonal entries
\[M_{ii}=d+2(i-1).\]
Hence we can find $q(x)y\in V$ with $L(q(x)y)=U(\PP_{\omega_{-1},\infty}(\eta))$. By arguments analogous to the above, considering the cases of $d$ even and odd seperately,
we have
\[\|\det(M)\|_p\geq \frac{1}{d(
d+m)}p^{-2-m/(p-1)}.\] 
Let $M_j$ be the matrix $M$ with the $j$th column replaced by the coefficients of 
$U(\PP_{\omega_{-1},\infty}(\eta))$. The coefficients of $M$ are integral and so
\[\|\det(M_{j})\|_p\leq \|U(\PP_{\omega_{-1},\infty}(\eta))\|_p.\] 
Consequently, the coefficients of $q(x)$ are bounded above by
\[d(d+m)p^{2+m/(p-1)}\|U(\PP_{\omega_{-1},\infty}(\eta))\|_p.\]
The bound on the coefficients of
$T(\PP_{\omega_{-1},\infty}(\eta-\eta'))$ follows from Lemma~\ref{l:ppatinfty}.
\end{proof}

\begin{algorithm}[H] 
\label{PoleRedInfinity}
\caption{\bf Pole reduction at infinity}

\KwIn{An odd meromorphic form $\eta$ such that $\frac{\eta}{\omega_{0}}$ has degree $d-2+m$ for some positive integer $m$.}
\KwOut{A function $F$ such that the $\frac{\eta-dF}{\omega_{0}}$ has degree at most $d-2$.}

\begin{enumerate}
    \item For $j=0,\dots,m-1$, compute the expansions
    \[d(x^jy) = \left(w_{d-1+j}x^{d-1+j} + \text{lower order terms}\right)\frac{dx}{2y}.\]
    \item Until $\frac{\eta}{\omega_0}$ has degree at most $d-2$, do the following:
    \begin{enumerate}
        \item Compute the expansion \[\eta = \left(u_{d-1+j}x^{d-1+j} + \text{lower order terms}\right)\frac{dx}{2y}.\]
        \item Define $a_{j} = u_{d-1+j}/w_{d-1+j}$ and set $\eta = \eta - a_{j}d(x^jy)$.
        \end{enumerate}
    \item Return \[F=(a_{m-1}x^{m-1}+a_{m-2}x^{m-2}+\dots+a_0)y.\]
\end{enumerate}
\end{algorithm}

\smallskip

We can now apply the algorithms described above to find a primitive of $\eta$. We first subtract exact forms from $\eta$ to remove the non-simple poles over non-roots of $g(x)$ and to remove the poles over roots of $g(x)$. Then, we reduce the pole order at $\infty$. Because the exact forms only affect the principal parts of one finite point at a time, we have the following: 

\begin{prop} \label{p:primitive}
Let $f_{\omega_0},\dots,f_{\omega_{d-2}}$ and $f_{\nu_1},\dots,f_{\nu_k}$ be Coleman primitives of $\omega_1,\dots,\omega_{d-2}$ and $\nu_1,\dots,\nu_k$, respectively.  Let $\eta$ be an odd $1$-form on $X$ such that $\frac{\eta}{\omega_0}$ has poles at points $\{\beta_1,\dots,\beta_{\ell},\infty\}\subset \P^1(\bK)$ of order $n_1,\dots,n_\ell,n_\infty$. Suppose $\beta_1,\dots,\beta_k$ are not roots of $g(x)$ and $\beta_{k+1},\dots,\beta_{\ell}$ are roots of $g(x)$. Moreover, we will suppose $\|\beta_i\|_p\leq 1$ for all $i$, $\|g(\beta_i)\|_p=1$ for $i=1,\dots,k$ and $\|g'(\beta_i)\|_p=1$ for $i=k+1,\dots,\ell$. Let $\alpha_i\in \pi^{-1}(\beta_i)$. Then $\eta$ has a Coleman primitive that is a linear combination of the following:
\begin{enumerate}
    \item $\frac{y}{(x-\beta_i)^j}$ where $1\leq j\leq n_i-1$  for $i=1,\dots,k$ with coefficient with norm at most 
    \[p^{n_i/(p-1)}\|\PP_{\omega_0,\alpha_i}(\eta)\|_p,\]
    \item $\frac{y}{(x-\beta_i)^j}$ where $1\leq j\leq n_i$ for $i=k+1,\dots,\ell$ with coefficient with norm at most
    \[n_ip^{1+n_i/(p-1)}\|\PP_{\omega_0,\alpha_i}(\eta)\|_p,\]
    \item $x^jy$ where $0\leq j\leq \max(n_\infty-d+2,2)$ with coefficient with norm at most the maximum of the following:
    \begin{enumerate}
       \item  $d(d+n_\infty)p^{2+n_\infty/(p-1)}\|U(\PP_{\omega_{-1},\infty}(\eta))\|_p,$
       \item  $\max_{i=1,\dots,k}\left(p^{n_i/(p-1)}\|\PP_{\omega_0,\alpha_i}(\eta)\|_p\right)$, and
       \item $\max_{i=k+1,\dots,\ell}\left(n_ip^{1+n_i/(p-1)}\|\PP_{\omega_0,\alpha_i}(\eta)\|_p\right).$
    \end{enumerate}
    \item $f_{\nu_i}$ for $i=1,\dots,k$ with coefficient equal to 
    \[\frac{\Res_{\alpha_i}(\eta)}{\Res_{\alpha_i}(\nu_i)},\]
    \item $f_{\omega_i}$ for $i=0,\dots,d-2$ with coefficient with norm at most 
    \[p^{2+\max(n_i/(p-1),n_\infty/(p-1))}\max(d(d+n_\infty),n_i)
    \max(\|\PP_{\omega_0,\alpha_i}(\eta)\|_p,\|\PP_{\omega_{-1},\infty}(\eta)\|_p)\] 
    where the maximum is taken over $\alpha\in\pi^{-1}(\{\beta_1,\dots,\beta_\ell\})$.
\end{enumerate}
\end{prop}

\begin{proof}
Let $\eta$ be an odd $1$-form on $X$ such that $\frac{\eta}{\omega_0}$ has poles at points $\{\beta_1,\dots,\beta_{\ell}\}\subset \A^1(\bK)$ of order $n_1,\dots,n_\ell$. Then we apply the Lemma~\ref{polereductionnonWss} and Lemma~\ref{polereductionWss} to reduce the pole orders at the $\beta$'s. Because the exact forms that we subtract for one $\beta_i$ does not affect the principal parts at other $\beta_i$'s, the pole order reduction steps are independent, and we have the above bounds on coefficients. These operations do affect the principal parts above $\infty$ in degrees up to $d-2$ according to the bounds in the lemmas. This leads to the bounds for the coefficient of $f_{\omega_i}$.
\end{proof}

\begin{algorithm}[H] 
\caption{\bf Cohomology class for a meromorphic $1$-form}

\KwIn{A meromorphic form 
\[\eta = x^{n_{\infty}}\prod_{j=1}^{\ell}\frac{1}{(x-\beta_j)^{n_j}}\frac{dx}{2y}\] 
for nonnegative integers $n_1,\dots,n_\ell,n_{\infty}$.}
\KwOut{A function $F$ and constants $c_i,d_j$ such that
\[\eta=dF+\sum_{i=0}^{d-2} c_i\omega_i + \sum_{j=1}^k d_j\nu_j.\]}

\begin{enumerate}
   \item \emph{Constants $d_j$:} For $j=1,\dots,k$, pick $\alpha_j\in\pi^{-1}(\beta_j)$ and compute
   \[d_j=\frac{\Res_{\alpha_j}(\eta)}{\Res_{\alpha_j}(\nu_j)}.\]
   \item \emph{Non-Weierstrass points:} Using Algorithm~\ref{PoleRednonWeierstrass}, find a function $F_{nw}$ such that the form $\eta-dF_{nw}$ has at worst simple poles at points above $\beta_1,\dots,\beta_k$.
   \item \emph{Finite Weierstrass points:} Using Algorithm~\ref{PoleRedWeierstrass}, find a function $F_w$ such that the form $\eta-dF_{nw}-dF_w$ is regular at points above $\beta_{k+1},\dots,\beta_\ell$.
   \item \emph{The point(s) at infinity:} Using Algorithm~\ref{PoleRedInfinity}, find a function $F_\infty$ such that $\frac{\eta-dF_{nw}-dF_w-dF_\infty}{\omega_0}$ has degree at most $d-2$.
   \item \emph{Constants $c_i$:}  Compute
   \[\frac{\eta-dF_{nw}-dF_w-dF_\infty-\sum d_j\nu_j}{\omega_0}=c_0+c_1x+\dots+c_{d-2}x^{d-2}.\]
   \item Return
   \[F=F_{nw}+F_w+F_\infty,\ \{c_i\}_i,\ \{d_j\}_j.\]
   \end{enumerate}
\end{algorithm}

\section{Power Series Expansion} \label{s:powerseriesexpansion}

We will write a power series expansion of $1$-forms $x^i\frac{dx}{2y}$ on basic wide opens in a semistable covering of a hyperelliptic curve $\pi\colon X^{\an}\to \P^{1,\an}$ defined by $y^2=f(x)$ following the methods in Section~\ref{HyperellipticCoverings}.  We will suppose that $f(x)$ is a monic polynomial with integral coefficients in some finite extension of $\Q_p$ and, moreover, that the roots of $f(x)$ lie in a field $\bK$ of ramification degree $e$ over $\Q_p$. By our assumptions, these roots have $p$-adic valuation at most $1$. Let $S_f$ be the set of roots of $f(x)$. Let $U$ be an element of a good semistable covering of $\P^{1,\an}$ with respect to $S_f$. We have an embedding $\iota\colon U\to \P^{1,\an}$ such that the points of $\iota(S_f\cap U(\C_p))$ lie in distinct residue discs. We will use $x$ to denote the coordinate on $\A^1\subset\P^1$.  Without loss of generality, we may suppose that $U$ is the open disc $B(0,R)$ (for some $R>1$) minus some closed discs and that $S_f\cap U(\C_p)\subset \overline{B}(0,1)$. Let $I_{\infty}$ be the set of roots of $f(x)$ lying outside of $B(0,R)$. Because the roots of $f(x)$ are $\bK$-points, the elements of $I_\infty\setminus\{\infty\}$ have norm at least $p^{1/e}$. We partition the roots of $f(x)$ in $B(0,R)$ by residue disc: $S_f\cap B(0,R) = \cup_{j=1}^{m}I_j$. Notice that some of $I_j$'s may have only one element. We can relabel these sets such that
\begin{enumerate}
\item for $j=1,\dots,k$, $|I_j|\geq 2$ and $|I_j|$ is even;
\item for $j=k+1,\dots,\ell$, $|I_j|\geq 2$ and $|I_j|$ is odd; and
\item for $j=\ell+1,\dots,m$, $|I_j|=1$.
\end{enumerate}
For $j=1,\dots,\ell$, pick $\beta_j\in\A^1(\bK)\setminus U(\bK)$ in the same residue disc as the points in $I_j$ (we may even take $\beta_j$ to be an element of $I_j$); and for $j=\ell+1,\dots,m$, let $\beta_j$ denote the unique element of $I_j$. Notice that $\|\beta_j\|_p \leq 1$ for all $j$. Define 
\[ L_j =
  \begin{cases}
    \hfil |I_j|/2       &\text{for } j=1,\dots,k;\\
    (|I_j|-1)/2  &\text{for } j=k+1,\dots,\ell
  \end{cases} \] and set
\begin{align*}
g(x) &= \prod_{j=k+1}^{m} (x-\beta_j),\\
h(x) & = \prod_{j=1}^{\ell}(x-\beta_j)^{L_j},\\
k(x) &= \left(\prod_{j=1}^{\ell}\prod_{\beta\in I_j}\left(\frac{x-\beta}{x-\beta_j}\right)\right)\left(\prod_{\beta\in I_\infty\setminus\{\infty\}}(x-\beta)\right).
\end{align*} 

Observe that $f(x)=g(x)h(x)^2k(x)$. Since $g(x)$ has at most one root in each residue disc and for $j=1,\dots,k$, the element $\beta_j$ is not in the same residue disc as a root of $g(x)$, we have
\begin{align*}
\|g(\beta_j)\|_p &= 1 \text{ for } j=1,\dots,k,\\
\|g'(\beta_j)\|_p &= 1 \text { for } j=k+1,\dots,\ell.
\end{align*}

Set 
\[\tilde{y}=\frac{y}{h(x)k(x)^{1/2}}.\]
Note that $\frac{1}{h(x)k(x)^{1/2}}$ is an analytic function on $U$ by construction; so $\tilde{y}^2=g(x)$ for $x\in  U$ defines a union of at most two basic wide opens in $X^{\an}$. Write $\tilde{X}$ for the complete curve defined $\tilde{y}^2=g(x)$. We may write $\tilde{X}_{g(x)}$ for $\tilde{X}$ when the polynomial $g(x)$ needs to be specified. 

We have
\[\omega_i=x^i\frac{dx}{2y}=\frac{x^i}{h(x)k(x)^{1/2}}\frac{dx}{2\tilde{y}}.\]
We will expand $\omega_i$ in a power series on $\pi^{-1}(U)$. We may write
\begin{align*}
k_j(x) &= \prod_{\beta\in I_j}\Big(1-\frac{\beta-\beta_j}{x-\beta_j}\Big) \text{ for } j=1,\dots,\ell, \\ 
k_\infty(x) &= \prod_{\beta\in I_\infty\setminus\{\infty\}}(-\beta)(1-\beta^{-1}x),
\end{align*} 
so $k(x)=\big(\prod_j k_j(x)\big)k_\infty(x)$. Now,
\begin{align*}
\frac{1}{k_j(x)^{1/2}} &= \prod_{\beta\in I_j}\Big(1-\frac{\beta-\beta_j}{x-\beta_j}\Big)^{-1/2}=\sum_{n=0}^{\infty}\frac{B_{jn}}{(x-\beta_j)^n} \\ 
\frac{1}{k_\infty(x)^{1/2}} &= \prod_{\beta\in I_\infty\setminus\{\infty\}}
(-\beta)^{-1/2}(1-\beta^{-1}x)^{-1/2}=\sum_{n=0}^{\infty}B_{\infty n}x^n
\end{align*} 
for some $B_{jn}$'s and $B_{\infty n}$'s. Then,
\begin{align*}
\frac{x^i}{h(x)k(x)^{1/2}} &= \Bigg(\prod_{j=1}^\ell \sum_{n=0}^{\infty}\frac{B_{jn}}{(x-\beta_j)^{n+L_j}}\Bigg)\Bigg(\sum_{n=0}^{\infty}B_{\infty n}x^{n+i}\Bigg)\\ 
&= \sum\limits_{\substack{n_1\geq L_1,\dots,n_{\ell}\geq L_{\ell} \\ n_\infty\geq i}}\Bigg(B_{n_1,\dots,n_{\ell},n_\infty}x^{n_{\infty}}\prod_{j=1}^{\ell}\frac{1}{(x-\beta_j)^{n_j}}\Bigg)
\end{align*} 
for some $B_{n_1,\dots,n_{\ell},n_\infty}$'s. We may bound these coefficients as follow.

\begin{prop} There is a constant $C$ such  that \[\|B_{n_1,\dots,n_{\ell},n_{\infty}}\|_p\leq Cp^{-{\frac{(n_\infty-i)+\sum_{j=1}^{\ell}(n_j-L_j)}{e}}}.\]
\end{prop}

\begin{proof} First observe that because $p\neq 2$, the coefficients of 
\[(1-y)^{-1/2}=\sum_{n=0}^\infty \frac{1}{2^{2n}}\binom{2n}{n}y^n\]
are $p$-adic integers.
Since $\|\beta-\beta_j\|_p\leq p^{-1/e}$ for each $\beta\in I_j$, by the ultrametric triangle inequality, we have $\|B_{jn}\|_p\leq C_jp^{-n/e}$ for some constant $C_j$. By an identical argument, 
we have 
\[\|B_{\infty n}\|_p\leq C_\infty p^{-n/e}\] 
for some constant $C_\infty$. By multiplying together our inequalities, we get the desired conclusion.
\end{proof}

Consequently, the expression
\begin{equation} \label{e:omegaiexpansion}
\omega_i= \sum\limits_{\substack{n_1\geq L_1,\dots,n_{\ell}\geq L_{\ell} \\ n_\infty\geq i}}\Bigg(B_{n_1,\dots,n_{\ell},n_\infty}x^{n_{\infty}}\prod_{j=1}^{\ell}\frac{1}{(x-\beta_j)^{n_j}}\frac{dx}{2\tilde{y}}\Bigg)
\end{equation}
makes sense. 

\begin{prop} Let $\tilde{\omega}_i=x^i\frac{dx}{2\tilde{y}}$. Then,
\begin{align*}
\left\|\PP_{\tilde{\omega}_0,\alpha}\left(x^{n_{\infty}}\prod_{j=1}^{\ell}\frac{1}{(x-\beta_j)^{n_j}}\frac{dx}{2\tilde{y}}\right)\right\|_p&\leq 1,\\
\left\|\PP_{\tilde{\omega}_{-1},\infty}\left(x^{n_{\infty}}\prod_{j=1}^{\ell}\frac{1}{(x-\beta_j)^{n_j}}\frac{dx}{2\tilde{y}}\right)\right\|_p&\leq 1
\end{align*}
where $\alpha$ is a point over any $\beta_i.$ 
\end{prop}

\begin{proof} 
If $\alpha$ is a point over some $\beta_i$, set $t=x-\beta_i$. Then, we have the following bounds on the coefficients of these power series, considered as Laurent series in $t$:
\begin{eqnarray*}
\left\|\frac{1}{(x-\beta_i)^{n_i}}\right\|_p&=& 1,\\
\|x^{n_\infty}\|_p&=& 1,\\
\left\|\frac{1}{(x-\beta_j)^{n_j}}\right\|_p&\leq& 1,\ \ j\neq i.\\
\end{eqnarray*}
Here, the last inequality follows from the observation that $\|\beta_i-\beta_j\|_p\geq 1$ for $i \neq j$. Because the Gauss norm $\|\cdot\|_p$ is multiplicative, for Laurent series $f$ and $g$,
\[\|\PP(fg)\|_p\leq \|fg\|_p =  \|f\|_p\|g\|_p\]
from which the conclusion follows.
An analogous arguments holds for $\infty$ using $t=1/x$.
\end{proof}

\begin{prop} Suppose $e<p-1$ and set $r = \frac{1}{e} - \frac{1}{p-1}$. Pick $R$ with $1<R<p^r$ and let 
\[D=\overline{B}(0,R)\setminus\bigcup_{i=1}^{\ell} B(\beta_i,1/R).\]
The $1$-form 
\[\omega_i=\frac{x^i}{h(x)k(x)^{1/2}}\frac{dx}{2\tilde{y}}\] 
has a Coleman primitive on $\pi^{-1}(D)$ given as the sum of terms of the following form:
\begin{enumerate}
    \item \label{i:firstterm} $a_{ij}\frac{y}{(x-\beta_i)^j}$ for $i=1,\dots,\ell$ and $j=1,2,\dots$,
    \item $b_j x^jy$ for $j=0,1,\dots$, 
    \item $c_i f_{\omega_i}$ for $i=0,\dots,d-2$, and
    \item \label{i:lastterm} $d_i f_{\nu_i}$ for $i=1,\dots,k$,
\end{enumerate}
where $f_{\omega_i}$ and $f_{\nu_i}$ are as in Proposition~\ref{p:primitive}. 
\end{prop}

\begin{proof}
The coefficients for the power series expansion of $\omega$ in \eqref{e:omegaiexpansion} decay at the rate of $p^{-N/e}$ where $N=n_\infty+\sum_{j=1}^{\ell} n_j$. By examining the summands,
we see that they converge uniformly on $D$.
On the other hand, by Proposition~\ref{p:primitive}, the coefficients of the primitives of the summands (expressed in terms of the form \eqref{i:firstterm}-\eqref{i:lastterm}) grow slower than $Np^{N/(p-1)}$. We call the series produced by integrating the power series expansion, \defi{the series of primitives}. For any given term of the type \eqref{i:firstterm}-\eqref{i:lastterm}, its coefficient is convergent in the series of primitives.

We have to verify the hypotheses of Proposition~\ref{swapintsum} to  show that the sum of the series of primitives is equal to the primitive of the sum of the series. By construction, the series of primitives is locally uniformly convergent. Moreover, for any lift of Frobenius $\phi$ and annihilating polynomial $P$, $P(\phi^*)$ applied to the series of primitives converges uniformly on a Frobenius neighborhood within $D$.  Moreover, the restriction of the  series of primitives to boundary annuli is uniformly convergent.
\end{proof}

\begin{rem}
While we employ the algorithms from \cite{BBComputing} to integrate the $1$-forms $\nu_j$ on hyperelliptic basic wide opens, we can formulate a different integration algorithm similar to the work of Tuitman \cite{Tuitman:counting1,Tuitman:counting2} using our techniques. Specifically, we can pick a lift of Frobenius $\phi$ on a hyperelliptic basic wide open. By replacing the lift of Frobenius by some power, we can ensure that it preserves the residue discs containing $\beta_1,\dots,\beta_{\ell}$. Consequently, $\phi^*\nu_j$ can be written as a power series as in this section. By using the techniques of the previous section, $\phi^*\nu_j$ can be rewritten as a linear combination of $1$-forms $\{\omega_0,\dots,\omega_{d-2},\nu_1,\dots,\nu_k\}$ and an exact form $dh_j$. Consequently, one obtains a matrix representing the action of Frobenius on odd cohomology and uses it to write down $p$-adic integrals. We will explore this in future work. 
\end{rem}

\section{Integration on Curves}

\subsection{Berkovich--Coleman Integration on Paths}

We explain how to perform Berkovich--Coleman integrals on a hyperelliptic curve $X^{\an}$. Such an integral is to be done along a path $\gamma$ in $X^{\an}$. We will break up the path into smaller paths lying in hyperelliptic basic wide opens. Fix a holomorphic $1$-form $\omega$ on $X^{\an}$.

Let $\cC$ be a semistable covering of $X^{\an}$ with dual graph $\Gamma$. For a vertex $v$ of $\Gamma$, let $U_v$ be the corresponding element of the covering. For $e=vw$, let $U_e$ be the corresponding component of the intersection $U_v\cap U_w$. Pick a point $P_v$ in each $U_v$ and a point $P_e$ in each $U_e$. These are called \textbf{reference points}. For each oriented edge $e$, write $i(e)$ and $t(e)$ for the initial and terminal point of $e$, respectively.

To a path $\gamma=e_1e_2\dots e_{\ell}$ in $\Gamma$ from $v$ to $w$, we can consider the Berkovich--Coleman integral of $\omega$ from the reference point $P_v$ to the reference point $P_w$ along the path $\gamma$. In fact, because $\Gamma$ is identified with the skeleton of $X^{\an}$, there is a unique path $\tilde{\gamma}_{vw}$ in $X^{\an}$ from $P_v$ to $P_w$ (up to fixed endpoint homotopy) whose image under $\tau\colon X^{\an} \to\Gamma$ is $\gamma$. We have
\[\BCint_{\tilde{\gamma}_{vw}} \omega = 
\sum_{i=1}^{\ell} \left(
\BCint_{P_{i(e_i)}}^{P_{e_i}} \omega
+\BCint_{P_{e_i}}^{P_{t(e_i)}} \omega
\right).\]
Here the integral from $P_{i(e_i)}$ to $P_{e_i}$ is to be performed on $U_{i(e_i)}$ and the integral from $P_{e_i}$ to $P_{t(e_i)}$ is to be performed on $U_{t(e_i)}$. Indeed, we can see the path $\tilde{\gamma}_{vw}$ as the concatenation (over $i$) of the path from $P_{i(e_i)}$ to $P_{e_i}$ in $U_{i(e_i)}$ followed by the path from $P_{e_i}$ to $P_{t(e_i)}$ in $U_{t(e_i)}$. 

Now, given $x\in U_v$, $y\in U_w$ and a path $\gamma$ from $v$ to $w$ in $\Gamma$, we may consider the Berkovich--Coleman integral of $\omega$ from $x$ to $y$ along $\gamma$. Indeed, it is the integral along any path $\tilde{\gamma}$ from $x$ to $y$ tropicalizing to $\gamma$:
\[\BCint_{\tilde{\gamma}} \omega=\BCint_x^{P_v} \omega+
\BCint_{\tilde{\gamma}_{vw}} \omega + 
\BCint_{P_w}^y \omega\]
where the first and last integrals on the right side are performed on $U_v$ and $U_w$, respectively. This integral is independent of the choices of reference points.

Finally, for a closed path $\gamma$ in $\Gamma$ at a vertex $v$, we may consider the \defi{Berkovich--Coleman period} 
\[\BCint_{\tilde{\gamma}} \omega = \BCint_{\tilde{\gamma}_{vv}} \omega.\]
Again, this is independent of the choice of reference points. Indeed, it depends only on the homology class of $\gamma$. 

This gives the following algorithm for performing Berkovich--Coleman integration of $\omega$. In particular, we can compute the periods of $\omega$ around closed loops.

\smallskip

\begin{algorithm}[H] \label{ComputingBCIntegrals}
\caption{\bf Computing Berkovich--Coleman integrals}

\KwIn{
\begin{itemize}
    \item A holomorphic $1$-form $\omega$ on $X^{\an}$.
    \item Points $x\in U_v, y\in U_w$.
    \item A path $\gamma=e_1e_2\dots e_{\ell}$ from $v$ to $w$ in $\Gamma$.
\end{itemize}
}
\KwOut{Berkovich--Coleman integral of $\omega$ from $x$ to $y$ along $\tilde{\gamma}$.}

\begin{enumerate}
    \item \label{BCintegralfororientededge} For each $i$, compute the integrals
    \[\BCint_{P_{i(e_i)}}^{P_{e_i}} \omega,\  \BCint_{P_{e_i}}^{P_{t(e_i)}} \omega\] on the basic wide opens $U_{i(e_i)}$ and $U_{t(e_i)}$, respectively.
    \item Compute the sum 
    \[\BCint_{\tilde{\gamma}_{vw}} \omega =  \sum_{i=1}^{\ell} \left(\BCint_{P_{i(e_i)}}^{P_{e_i}} \omega+\BCint_{P_{e_i}}^{P_{t(e_i)}} \omega\right).\]
    \item Compute the integrals
    \[\BCint_x^{P_v}\omega,\  \BCint_{P_w}^y \omega\]
    on the basic wide opens $U_v$ and $U_w$, respectively.
    \item Return 
    \[\BCint_{\tilde{\gamma}} \omega = \BCint_x^{P_v}\omega + \BCint_{\tilde{\gamma}_{vw}} \omega + \BCint_{P_w}^y \omega.\]
\end{enumerate} 
\end{algorithm}

\subsection{Abelian Integration}

We have an algorithm for computing abelian integrals on a hyperelliptic curve $X$ using Theorem~\ref{thm:abintformula} given a semistable cover $\cC$ and its dual graph $\Gamma$. 

\smallskip

\begin{algorithm}[H] 

\caption{\bf Computing abelian integrals}

\KwIn{
\begin{itemize}
    \item A holomorphic $1$-form $\omega$ on $X$.
    \item Points $x,y\in X(\bK)$.
\end{itemize}
}
\KwOut{Abelian integral of $\omega$ from $x$ to $y$.}

\begin{enumerate}
    \item Pick a path $\gamma$ in $\Gamma$ from $v$ to $w$ for vertices $v,w$ such that $x\in U_v$ and $y\in U_w$.
    \item Compute the Berkovich--Coleman integral
    \[\BCint_{\tilde{\gamma}} \omega\]
    as in Algorithm~\ref{ComputingBCIntegrals}.
    \item Pick a basis $C_1,\dots,C_h$ for $H_1(\Gamma;\Z)$ and a basis $\eta_1,\dots,\eta_h$ of $\Omega^1_{\trop}(\Gamma)$ dual to $C_1,\dots,C_h$ (see Remark~\ref{hombasistropbasis}).
    \item For each $i$, pick a loop $\gamma_i$ in $X^{\an}$ whose homology class is $C_i$.
    \item Compute the Berkovich--Coleman periods
    \[\BCint_{\gamma_i} \omega,\ i=1,\dots,h\] 
    as in Algorithm~\ref{ComputingBCIntegrals}.
    \item Compute the tropical integrals 
    \[\tint_{\gamma} \eta_i,\ i=1,\dots,h.\]
    \item Return
    \[\Abint_x^y \omega = \BCint_{\tilde{\gamma}}\omega-\sum_i \left(\BCint_{\gamma_i} \omega\right)\left(\tint_{\gamma} \eta_i\right).\]
\end{enumerate} 
\end{algorithm}

\begin{rem} \label{hombasistropbasis}
A basis of $H_1(\Gamma;\Z)$ and a dual tropical basis can be obtained from the tree $T$ as in Proposition~\ref{p:iota}. Let $C'_1,\dots,C'_h$ be a basis of $H_1(T_o,V_e;\Z)$ and let $D'_1,\dots,D'_h$ be a dual basis with respect to $\langle\cdot,\cdot\rangle$. Let $C_i=\iota(C'_i)$ and $D_i=\frac{1}{2}\iota(D'_i)$ where $\iota$ is given in Proposition~\ref{p:iota}.  Then, by Proposition~\ref{p:iotacycle}, $\{C_i\}$ and $\{\eta_i=\eta_{D_i}\}$ form a basis of $H_1(\Gamma;\Z)$ and a dual tropical basis of $1$-forms on $\Gamma$, respectively. 
\end{rem} 

\section{Numerical Examples}

Here, we illustrate our methods with numerical examples computed in Sage \cite{sagemath}. But first, we make the following remarks:

\begin{itemize}
    \item \emph{Sage restriction.} Let $X$ be a curve defined over $\Q_p$. An abelian integral on $X$ between $\Q_p$-rational points is an element of $\Q_p$. In our approach, such an integral is expressed as a sum of other integrals, each of which is an element of a possibly different finite extension of $\Q_p$. More precisely, reference points corresponding to edges might lie in highly ramified extensions and taking square roots  might force us to work with unramified extensions. In Sage, one can define these extensions individually, however, conversion between $p$-adic extensions has not been implemented yet. In order to deal with this restriction, in each of our examples, all computations will take place in a single extension.
    \item \emph{Weierstrass endpoints.} Let $X$ be an odd degree hyperelliptic curve with the Abel--Jacobi map $\iota\colon X\to J$ with base-point $\infty$. For Weierstrass points $R,S\in X(\C_p)$, the class $[S]-[R]$ represents a $2$-torsion point of $J(\C_p)$ since 
    \[\operatorname{div}(x-\alpha) = 2(\alpha,0)-2\infty\] 
    for any root $\alpha$ of the polynomial defining $X$. This implies by Remark~\ref{torsionvanish} that the abelian integrals with Weierstrass endpoints must vanish. We will observe this vanishing numerically to test the correctness of our algorithm. 
    \item \emph{Branch of logarithm.} As we discussed before, the Berkovich--Coleman integration requires a branch of the $p$-adic logarithm. We pick the Iwasawa branch, i.e., the one characterized by $\Log(p)=0$. Abelian integration does not depend on this choice.
\end{itemize}

In the examples below, as usual, $\omega_i$ will denote the holomorphic $1$-form $x^i\frac{dx}{2y}$ on the corresponding curve.

\begin{eg}(Genus~1)
Consider the elliptic curve $X/\Q$ \cite[\href{http://www.lmfdb.org/EllipticCurve/Q/272.b2}{272.b2}]{lmfdb} given by 
\[y^2=f(x)=(x-6)(x-5)(x+11).\] 
Its Mordell--Weil group is isomorphic to $\Z\times \Z/2\Z\times \Z/2\Z$ and the point $P=(-3,24)$ is a generator of the free part. 

Hereafter, we consider $X$ over 
the field $\Q_{17}$; clearly this curve has split multiplicative reduction. Set $R=(23,102), S=(7,6)$. Using the formal logarithm implementation in Sage \cite{sagemath}, one can easily check
\begin{equation}
\label{ECexample}
\Abint_S^R \omega_0 = 12 \cdot 17 + 8 \cdot 17^{2} + 15 \cdot 17^{3} + 9 \cdot 17^{4} + 16 \cdot 17^{5} + 8 \cdot 17^{6} + O(17^{7}).
\end{equation} 
We will compute this integral using our techniques and compare the results.

The set $\{U_1,U_2\}$ is a good semistable covering of $\P^{1,\an}$ with respect to $S_f=\{6,5,-11,\infty\}$ where
\[U_1 = \P^{1,\an}\setminus\overline{B}(6,1/17),\
U_2 = B(6,1)\] 
and we have the dual graphs $\Gamma$ and $T$
\[\begin{tikzpicture}
\draw [thick] (-4,0) ellipse (1cm and 0.4cm);
\filldraw (-5,0) circle (2pt);
\filldraw (-5.3,0) node{$v_1$};
\filldraw (-4,0.4) node{>};
\filldraw (-4,0.7) node{$e_1$};
\filldraw (-3,0) circle (2pt);
\filldraw (-2.7,0) node{$v_2$};
\filldraw (-4,-0.4) node{<};
\filldraw (-4,-0.7) node{$e_2$};

\qquad
 
\draw [thick] (0,0) -- (1.5,0);
\filldraw (0,0) circle (2pt);
\filldraw (0,0.4) node{$U_1$};
\filldraw (1.5,0) circle (2pt);
\filldraw (1.5,0.4) node{$U_2$};
\draw (0,0) -- (-0.3,-0.5);
\filldraw (-0.3,-0.7) node{\tiny{$5$}};
\draw (0,0) -- (0.3,-0.5);
\filldraw (0.3,-0.7) node{\tiny{$\infty$}};
\draw (1.5,0) -- (1.2,-0.5);
\filldraw (1.2,-0.7) node{\tiny{$6$}};
\draw (1.5,0) -- (1.8,-0.5);
\filldraw (1.8,-0.7) node{\tiny{$-11$}};
\end{tikzpicture}\] 
respectively. Note that $R\in\pi^{-1}(U_2),\ S\in \pi^{-1}(U_1).$

The cycle $C = e_1+e_2$ and the tropical $1$-form $\eta=\frac{1}{2}\eta_C$ are as in Corollary~\ref{l:dualbasis}. Now, we pick reference points. Let $P_{v_1}$ and $P_{v_2}$ be  points whose $x$-coordinates are $1$ and $-28$, respectively; hence $P_{v_i}\in \pi^{-1}(U_i)$. Let $P_{e_1}$ and $P_{e_2}$ denote the two different points whose $x$-coordinates are both $a+6$ where $a^2=17$. Notice that these points lie in the intersection $ \pi^{-1}(U_1)\cap \pi^{-1}(U_2)$. We assume that the point $P_{e_1}$ lies in the component corresponding to the edge $e_1$.

We have
\begin{align*}
\pi^{-1}(U_1) &\simeq \{(x,\tilde{y})\mid\tilde{y}^2=x-5, x\in U_1\} \\ 
&= \tilde{X}_{x-5}^{\an}\setminus D_1,\ D_1=\{(x,\tilde{y})\mid\tilde{y}^2=x-5, x\in \overline{B}(6,1/17)\}
\end{align*} where 
\[\tilde{y}=\frac{y}{\ell(x)},\ \ell(x) = (x-6)\Big(1+\frac{17}{x-6}\Big)^{1/2}.\] 
Define
\begin{eqnarray*}
\P^1&\to& \tilde{X}_{x-5}\\
T&\mapsto&\begin{cases}
            \hfil \infty    & \text{if } T=\infty,  \\
            (T^2+2T+6,T+1)    &    \text{otherwise.}
           \end{cases}
\end{eqnarray*} 
This is a parametrization and induces an isomorphism 
\[\P^{1,\an}\setminus\big(\overline{B}(0,1/17)\cup\overline{B}(-2,1/17)\big)\simeq \tilde{X}_{x-5}^{\an}\setminus D_1.\]
This annulus is isomorphic to a standard annulus by \begin{eqnarray*}
A(1/17^2,1)&\xrightarrow{\sim}& \P^{1,\an}\setminus\big(\overline{B}(0,1/17)\cup\overline{B}(-2,1/17)\big)\\
t&\mapsto&\frac{-2t}{t-17}
\end{eqnarray*} and, under these isomorphisms, the $1$-form ${\omega_0}_{\vert\pi^{-1}(U_1)}$ is represented on $A(1/17^2,1)$ by
\[\Big(1+\frac{(t-17)^2}{4t}\Big)^{-1/2}\frac{dt}{2t}.\]

Similarly, we have
\begin{align*}
\pi^{-1}(U_2) &\simeq \{(\tilde{x},\tilde{y})\mid\tilde{y}^2=\tilde{x}(\tilde{x}+1), \tilde{x}\in B(0,17)\} \\ 
&= \tilde{X}_{\tilde{x}(\tilde{x}+1)}^{\an}\setminus D_2,\ D_2=\{(\tilde{x},\tilde{y})\mid\tilde{y}^2=\tilde{x}(\tilde{x}+1), \tilde{x}\in \overline{B}(\infty,1/17)\}
\end{align*} where 
\[\tilde{x}=\frac{x-6}{17},\ \tilde{y}=\frac{y}{17\ell(\tilde{x})},\ \ell(\tilde{x})=(1+17\tilde{x})^{1/2}.\] 
Define
\begin{eqnarray*}
\P^1&\to& \tilde{X}_{\tilde{x}(\tilde{x}+1)}\\
T&\mapsto&\begin{cases}
            \hfil \infty^+    & \text{if } T=0,  \\
            \hfil \infty^-    & \text{if } T=\infty,  \\
            \Big(\frac{1}{2}\big(T+\frac{1}{4T}\big)-\frac{1}{2},\frac{1}{2}\big(T-\frac{1}{4T}\big)\Big)    &    \text{otherwise.}
           \end{cases}
\end{eqnarray*} 
This is a parametrization and induces an isomorphism 
\[\P^{1,\an}\setminus\big(\overline{B}(0,1/17)\cup\overline{B}(\infty,1/17)\big)\simeq \tilde{X}_{\tilde{x}(\tilde{x}+1)}^{\an}\setminus D_2.\] 
The annulus on the left is isomorphic to a standard annulus by 
\begin{eqnarray*}
A(1/17^2,1)&\xrightarrow{\sim}& \P^{1,\an}\setminus\big(\overline{B}(0,1/17)\cup\overline{B}(\infty,1/17)\big)\\
t&\mapsto& t/17
\end{eqnarray*}
and, under these isomorphisms, the $1$-form ${\omega_0}_{\vert\pi^{-1}(U_2)}$ is represented on $A(1/17^2,1)$ by
\[\Big(1+\frac{(t-17/2)^2}{2t}\Big)^{-1/2}\frac{dt}{2t}.\]

Let $\gamma$ be the concatenation of a path from $S$ to $P_{e_1}$ in $\pi^{-1}(U_1)$ and a path from $P_{e_1}$ to $R$ in $\pi^{-1}(U_2)$. Then $\tau(\gamma)=e_1$ and we have
\[\BCint_\gamma \omega_0 = 15 \cdot a^{4} + 11 \cdot a^{6} + 12 \cdot a^{8} + a^{10} + 11 \cdot a^{12} + O(a^{14}),\]
\[\tint_{e_1} \eta = \frac{1}{2}.\]

Consider the loop $\gamma_C=\gamma_1\gamma_2\gamma_3\gamma_4$ in $X^{\an}$ where $\gamma_1$ is a path from $P_{v_1}$ to $P_{e_1}$ in $\pi^{-1}(U_1)$, $\gamma_2$ is a path from $P_{e_1}$ to $P_{v_2}$ in $\pi^{-1}(U_2)$, $\gamma_3$ is a path from $P_{v_2}$ to $P_{e_2}$ in $\pi^{-1}(U_2)$ and $\gamma_4$ is a path from $P_{e_2}$ to $P_{v_1}$ in $\pi^{-1}(U_1)$. The paths are shown in a figure modified from one in \cite{krzb:utahsurvey}. 

\[\includegraphics[scale=0.12]{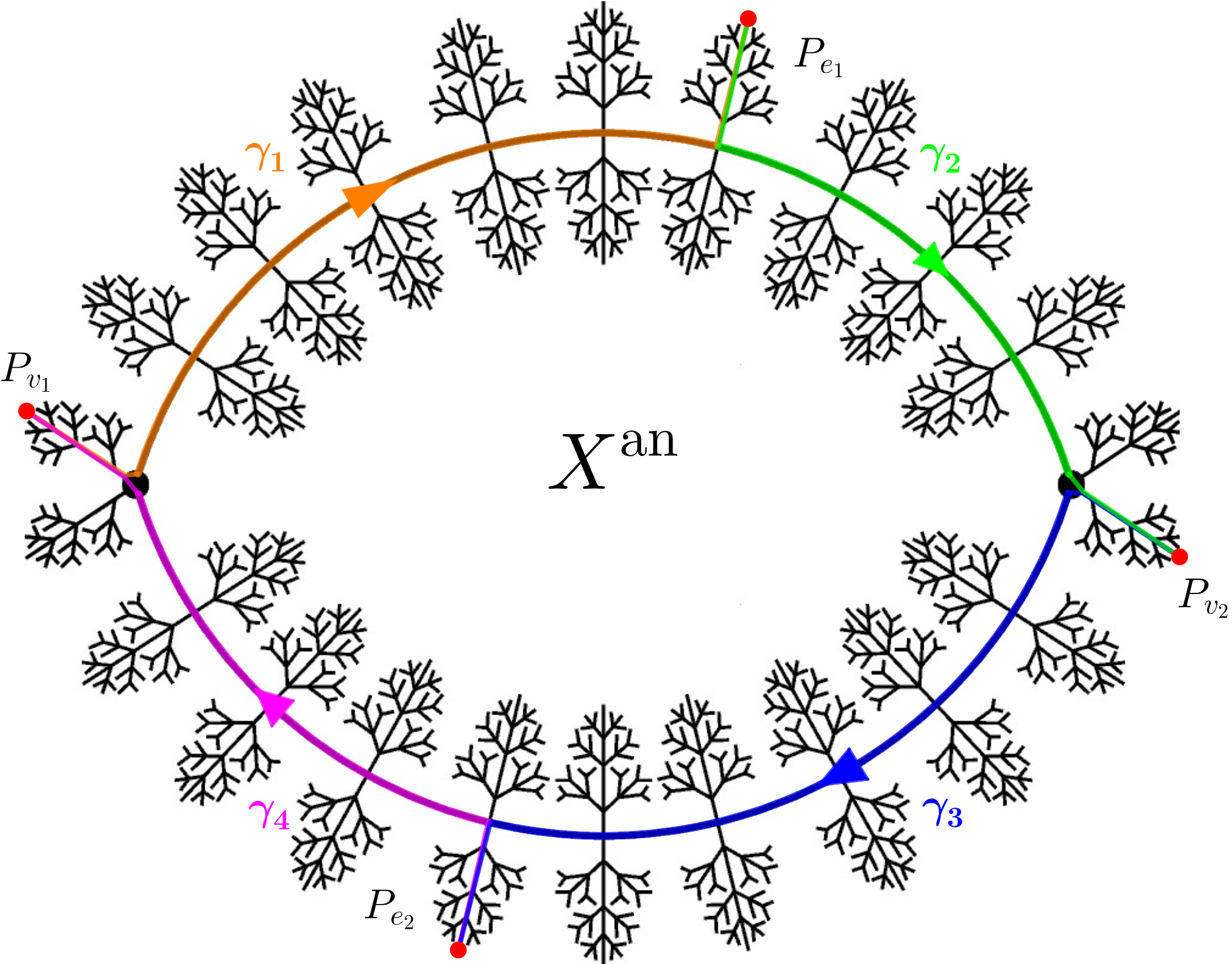}\]

The homology class of $\gamma_C$ is $C$ and the period integral is
\[\BCint_{\gamma_C} \omega_0 = 10 \cdot a^{2} + 12 \cdot a^{4} + 9 \cdot a^{6} + 5 \cdot a^{8} + 4 \cdot a^{10} + 4 \cdot a^{12} + O(a^{14}).\] 
Finally,
\begin{eqnarray*}
\Abint_S^R \omega_0 &=& \BCint_\gamma \omega_0 - \left(\BCint_{\gamma_C}\omega_0\right) \left(\tint_{e_1} \eta\right)\\
                    &=& 12 \cdot a^{2} + 8 \cdot a^{4} + 15 \cdot a^{6} + 9 \cdot a^{8} + 16 \cdot a^{10} + 8 \cdot a^{12} + O(a^{14})
\end{eqnarray*} 
which is the same result as in \eqref{ECexample} since $a^2 = 17$.

We also note that, using the addition law on elliptic curves, for each $i\in\{0,\dots,100\}$ our methods give
\[\Abint_{(5,0)+iP}^{(6,0)+iP} \omega_0 = \Abint_{(5,0)+iP}^{(-11,0)+iP} \omega_0 = O(a^{14})\]
demonstrating the vanishing of integrals between points whose difference is torsion.
\end{eg}

\begin{eg}(Genus~2)
Let $X/\Q_7$ be the genus $2$ curve defined by 
\[y^2=f(x)=x(x-1)(x-2)(x-3)(x-7).\] 
Set $R=(0,0), S=(1,0)$; we already know that the abelian integral of $\omega$ from $S$ to $R$ vanishes for every holomorphic form $\omega$ on $X$. Using our techniques, we will verify this up to a certain precision.

The set $\{U_1,U_2\}$ is a good semistable covering of $\P^{1,\an}$ with respect to $S_f=\{0,1,2,3,7,\infty\}$ where
\[U_1 = \P^{1,\an}\setminus\overline{B}(0,1/7),\
U_2 = B(0,1)\] 
and we have the dual graphs $\Gamma$ and $T$
\[\begin{tikzpicture}
\draw [thick] (-4,0) ellipse (1cm and 0.4cm);
\filldraw (-5,0) circle (2pt);
\filldraw (-5.3,0) node{$v_1$};
\filldraw (-4,0.4) node{>};
\filldraw (-4,0.7) node{$e_1$};
\filldraw (-3,0) circle (2pt);
\filldraw (-2.7,0) node{$v_2$};
\filldraw (-4,-0.4) node{<};
\filldraw (-4,-0.7) node{$e_2$};

\qquad
 
\draw [thick] (0,0) -- (1.5,0);
\filldraw (0,0) circle (2pt);
\filldraw (0,0.4) node{$U_1$};
\filldraw (1.5,0) circle (2pt);
\filldraw (1.5,0.4) node{$U_2$};
\draw (0,0) -- (-0.4,-0.5);
\filldraw (-0.4,-0.7) node{\tiny{$1$}};
\draw (0,0) -- (-0.15,-0.5);
\filldraw (-0.15,-0.7) node{\tiny{$2$}};
\draw (0,0) -- (0.15,-0.5);
\filldraw (0.15,-0.7) node{\tiny{$3$}};
\draw (0,0) -- (0.4,-0.5);
\filldraw (0.45,-0.7) node{\tiny{$\infty$}};
\draw (1.5,0) -- (1.2,-0.5);
\filldraw (1.2,-0.7) node{\tiny{$0$}};
\draw (1.5,0) -- (1.8,-0.5);
\filldraw (1.8,-0.7) node{\tiny{$7$}};
\end{tikzpicture}\] respectively. Notice that $R\in\pi^{-1}(U_2),\ S\in \pi^{-1}(U_1).$

The cycle $C = e_1+e_2$ and the tropical $1$-form $\eta=\frac{1}{2}\eta_C$ are as in Corollary~\ref{l:dualbasis}. Again, we pick reference points. Let $P_{v_1}$ and $P_{v_2}$ be points whose $x$-coordinates are $-1$ and $14$, respectively; hence $P_{v_i}\in \pi^{-1}(U_i)$. Let $P_{e_1}$ and $P_{e_2}$ denote the two different points whose $x$-coordinates are both $a$ where $a^2=7$. Notice that these points lie in the intersection $ \pi^{-1}(U_1)\cap \pi^{-1}(U_2)$. We assume that the point $P_{e_1}$ lies in the component corresponding to the edge $e_1$. 

The analytic open $\pi^{-1}(U_1)$ will be embedded into a good reduction elliptic curve. In fact, we have
\[\pi^{-1}(U_1)\simeq \{(x,\tilde{y})\mid\tilde{y}^2=(x-1)(x-2)(x-3), x\in U_1\}\] 
where 
\[\ \tilde{y}=\frac{y}{\ell(x)},\ \ell(x)=x\Big(1-\frac{7}{x}\Big)^{1/2}.\] 
In the new coordinates, the $1$-form ${\omega_i}_{\vert\pi^{-1}(U_1)}$ is given by \[\Big(1-\frac{7}{x}\Big)^{-1/2}x^{i-1}\frac{dx}{2\tilde{y}}.\]

For the other component, we have
\begin{align*}
\pi^{-1}(U_2) &\simeq \{(\tilde{x},\tilde{y})\mid\tilde{y}^2=\tilde{x}(\tilde{x}-1), \tilde{x}\in B(0,7)\} \\ 
&= \tilde{X}_{\tilde{x}(\tilde{x}-1)}^{\an}\setminus D,\ D=\{(\tilde{x},\tilde{y})\mid\tilde{y}^2=\tilde{x}(\tilde{x}-1), \tilde{x}\in \overline{B}(\infty,1/7)\}
\end{align*} where 
\[\tilde{x}=\frac{x}{7},\ \tilde{y}=\frac{y}{7\ell(\tilde{x})},\ \ell(\tilde{x})=\big((7\tilde{x}-1)(7\tilde{x}-2)(7\tilde{x}-3)\big)^{1/2}.\] 
Define
\begin{eqnarray*}
\P^1&\to& \tilde{X}_{\tilde{x}(\tilde{x}-1)}\\
T&\mapsto&\begin{cases}
            \hfil \infty^+    & \text{if } T=0,  \\
            \hfil \infty^-    & \text{if } T=\infty,  \\
            \Big(\frac{1}{2}\big(T+\frac{1}{4T}\big)+\frac{1}{2},\frac{1}{2}\big(T-\frac{1}{4T}\big)\Big)    &    \text{otherwise.}
           \end{cases}
\end{eqnarray*} 
This is a parametrization and induces an isomorphism \[\P^{1,\an}\setminus\big(\overline{B}(0,1/7)\cup\overline{B}(\infty,1/7)\big)\simeq \tilde{X}_{\tilde{x}(\tilde{x}-1)}^{\an}\setminus D.\] 
The annulus on the left is isomorphic to a standard annulus by \begin{eqnarray*}
A(1/7^2,1)&\xrightarrow{\sim}& \P^{1,\an}\setminus\big(\overline{B}(0,1/7)\cup\overline{B}(\infty,1/7)\big)\\
t&\mapsto& t/7
\end{eqnarray*} and, under these isomorphisms, the $1$-form ${\omega_i}_{\vert\pi^{-1}(U_2)}$ is can be expressed on $A(1/7^2,1)$ as
\[\Big(\frac{(t+7/2)^2}{2t}\Big)^i\bigg(\Big(\frac{(t+7/2)^2}{2t}-1\Big)\Big(\frac{(t+7/2)^2}{2t}-2\Big)\Big(\frac{(t+7/2)^2}{2t}-3\Big)\bigg)^{-1/2}\frac{dt}{2t}.\]

As in the previous example, take a path $\gamma$ from $S$ to $R$ such that $\tau(\gamma)=e_1$ and a take loop $\gamma_C$ whose homology class is $C$. Then our computations give
\begin{align*}
\BCint_\gamma \omega_0 &= 4 \cdot a^{6} + 2 \cdot a^{8} + 2 \cdot a^{10} + 5 \cdot a^{12} + O(a^{14}),\\
\BCint_\gamma \omega_1 &= 6 \cdot a^{2} + 6 \cdot a^{6} + 4 \cdot a^{10} + 6 \cdot a^{12} + O(a^{14}),\\
\BCint_{\gamma_C} \omega_0 &= a^{6} + 5 \cdot a^{8} + 4 \cdot a^{10} + 3 \cdot a^{12} + O(a^{14}),\\
\BCint_{\gamma_C} \omega_1 &= 5 \cdot a^{2} + a^{4} + 5 \cdot a^{6} + a^{8} + a^{10} + 6 \cdot a^{12} + O(a^{14}),
\end{align*} 
\[\tint_{e_1}\eta = \frac{1}{2}.\] 
Combining these, we get 
\[\Abint_S^R \omega_i = \BCint_\gamma \omega_i - \left(\BCint_{\gamma_C}\omega_i\right) \left(\tint_{e_1} \eta\right) = O(a^{14}),\ i=0,1\] 
from which our aim follows as every holomorphic $1$-form is a linear combination of $\omega_0$ and $\omega_1$.
\end{eg}

\begin{eg}(Genus~3)
Let $X/\Q_{13}$ be the genus $3$ curve given by 
\[y^2=f(x)=x(x-13)(x-169)(x-1)(x-14)(x-27)(x-4).\] 
The set $\{U_1,U_2,U_3,U_4\}$ is a good semistable covering of $\P^{1,\an}$ with respect to $S_f=\{0,13,169,1,14,27,4,\infty\}$ where 
\begin{align*}
U_1 &= \P^{1,\an}\setminus\big(\overline{B}(1,1/13)\cup \overline{B}(0,1/13)\big), \\
U_2 &= B(1,1), \\
U_3 &= A(1/169,1), \ U_4=B(0,1/13),
\end{align*} and we have the dual graphs $\Gamma$ and $T$
\[\begin{tikzpicture}
\draw [thick] (-6.5,0) -- (-5.5,0);
\draw [thick] (-5.5,0) -- (-4.5,0);
\draw [thick] (-3.75,0) ellipse (0.75cm and 0.4cm);
\filldraw (-4.5,0) circle (2pt);
\filldraw (-5.5,0) circle (2pt);
\filldraw (-6.5,0) circle (2pt);
\filldraw (-3,0) circle (2pt);
\filldraw (-6,0) node{>};
\filldraw (-5,0) node{>};
\filldraw (-3.75,0.4) node{>};
\filldraw (-3.75,-0.4) node{<};
\filldraw (-6,0.3) node{$e_1$};
\filldraw (-5,0.3) node{$e_2$};
\filldraw (-3.75,0.7) node{$e_3$};
\filldraw (-3.75,-0.7) node{$e_4$};
\filldraw (-6.5,-0.4) node{$v_2$};
\filldraw (-5.5,-0.4) node{$v_1$};
\filldraw (-4.6,-0.4) node{$v_3$};
\filldraw (-2.9,-0.4) node{$v_4$};

\qquad

\draw [thick] (-0.2,0) -- (3,0);
\filldraw (-0.2,0) circle (2pt);
\filldraw (1,0) circle (2pt);
\filldraw (2,0) circle (2pt);
\filldraw (3,0) circle (2pt);
\filldraw (-0.2,0.4) node{$U_2$};
\filldraw (1,0.4) node{$U_1$};
\filldraw (2,0.4) node{$U_3$};
\filldraw (3,0.4) node{$U_4$};
\draw (-0.2,0) -- (-0.6,-0.5);
\filldraw (-0.65,-0.7) node{\tiny{$1$}};
\draw (-0.2,0) -- (-0.2,-0.5);
\filldraw (-0.2,-0.7) node{\tiny{$14$}};
\draw (-0.2,0) -- (0.2,-0.5);
\filldraw (0.25,-0.7) node{\tiny{$27$}};
\draw (1,0) -- (0.8,-0.5);
\filldraw (0.75,-0.7) node{\tiny{$4$}};
\draw (1,0) -- (1.2,-0.5);
\filldraw (1.25,-0.7) node{\tiny{$\infty$}};
\draw (2,0) -- (2,-0.5);
\filldraw (2,-0.7) node{\tiny{$13$}};
\draw (3,0) -- (2.8,-0.5);
\filldraw (2.75,-0.7) node{\tiny{$0$}};
\draw (3,0) -- (3.2,-0.5);
\filldraw (3.25,-0.7) node{\tiny{$169$}};
\end{tikzpicture}\] respectively.

The cycle $C = e_3+e_4$ and the tropical $1$-form $\eta=\frac{1}{2}\eta_C$ are as in Corollary~\ref{l:dualbasis}. Let $P_{v_1},P_{v_2},P_{v_3},P_{v_4}$ be points whose $x$-coordinates are $2,20/7,-13/12,169/14$, respectively; hence $P_{v_i}\in \pi^{-1}(U_i)$. For an $a$ such that $a^4=13$, let $P_{e_1}$ and $P_{e_2}$ denote points whose $x$-coordinates are $a^2+1$ and $a^2$, respectively; and let $P_{e_3}$ and $P_{e_4}$ be the two different points whose $x$-coordinates are both $13a^2$. Notice that 
\[P_{e_1}\in \pi^{-1}(U_2)\cap \pi^{-1}(U_1),\ P_{e_2}\in \pi^{-1}(U_1)\cap \pi^{-1}(U_3)\] 
and that 
\[P_{e_3},P_{e_4}\in \pi^{-1}(U_3)\cap \pi^{-1}(U_4).\]
We assume that the point $P_{e_3}$ lies in the component corresponding to the edge $e_3$. 

The preimage $\pi^{-1}(U_1)$ will be embedded into a good reduction elliptic curve:
\[\pi^{-1}(U_1)\simeq \{(x,\tilde{y})\mid\tilde{y}^2=x(x-1)(x-4), x\in U_1\}\]
where
\[\tilde{y}=\frac{y}{\ell(x)},\ \ell(x) =x(x-1)\bigg(\Big(1-\frac{13}{x}\Big)\Big(1-\frac{169}{x}\Big)\Big(1-\frac{13}{x-1}\Big)\Big(1-\frac{26}{x-1}\Big)\bigg)^{1/2}.\] 
The $1$-form ${\omega_i}_{\vert\pi^{-1}(U_1)}$ is represented by 
\[\bigg(\Big(1-\frac{13}{x}\Big)\Big(1-\frac{169}{x}\Big)\Big(1-\frac{13}{x-1}\Big)\Big(1-\frac{26}{x-1}\Big)\bigg)^{-1/2}\frac{x^{i-1}}{x-1}\frac{dx}{2\tilde{y}}.\] 

Similarly, $\pi^{-1}(U_2)$ is also embedded into an elliptic curve:
\[\pi^{-1}(U_2) \simeq \{(\tilde{x},\tilde{y})\mid\tilde{y}^2=\tilde{x}(\tilde{x}-1)(\tilde{x}-2), \tilde{x}\in B(0,13)\}\]
where 
\[\textstyle \tilde{x}=\frac{x-1}{13},\ \tilde{y}=\frac{y}{13\sqrt{13}\cdot\ell(\tilde{x})},\ \ell(\tilde{x})=\big((13\tilde{x}+1)(13\tilde{x}-12)(13\tilde{x}-168)(13\tilde{x}-3)\big)^{1/2}.\] 
The $1$-form ${\omega_i}_{\vert\pi^{-1}(U_2)}$ becomes
\[\frac{1}{\sqrt{13}}(13\tilde{x}+1)^i\big((13\tilde{x}+1)(13\tilde{x}-12)(13\tilde{x}-168)(13\tilde{x}-3)\big)^{-1/2}\frac{d\tilde{x}}{2\tilde{y}}.\]

Now, $\pi^{-1}(U_3)$ will be embedded into a rational curve:
\[\pi^{-1}(U_3)\simeq \{(x,\tilde{y})\mid\tilde{y}^2=x-13, x\in U_3\}\] 
where 
\[\tilde{y}=\frac{y}{\ell(x)},\ \ell(x) =x\Big(1-\frac{169}{x}\Big)^{1/2}\big((x-1)(x-14)(x-27)(x-4)\big)^{1/2}.\]
Under this isomorphism,  the $1$-form ${\omega_i}_{\vert\pi^{-1}(U_3)}$ is represented by
\[\Big(1-\frac{169}{x}\Big)^{-1/2}\big((x-1)(x-14)(x-27)(x-4)\big)^{-1/2}x^{i-1}\frac{dx}{2\tilde{y}}.\]

The analytic open $\pi^{-1}(U_4)$ will also be embedded into a rational curve:
\[\pi^{-1}(U_4) \simeq \{(\tilde{x},\tilde{y})\mid\tilde{y}^2=\tilde{x}(\tilde{x}-1), \tilde{x}\in B(0,13)\}\]
where 
\[\textstyle \tilde{x}=\frac{x}{169},\ \tilde{y}=\frac{y}{169\cdot\ell(\tilde{x})},\ \ell(\tilde{x})=\big((169\tilde{x}-13)(169\tilde{x}-1)(169\tilde{x}-14)(169\tilde{x}-27)(169\tilde{x}-4)\big)^{1/2}.\] In the new coordinates,  the $1$-form ${\omega_i}_{\vert\pi^{-1}(U_4)}$ becomes \[\big((169\tilde{x}-13)(169\tilde{x}-1)(169\tilde{x}-14)(169\tilde{x}-27)(169\tilde{x}-4)\big)^{-1/2}(169\tilde{x})^i\frac{d\tilde{x}}{2\tilde{y}}.\]

We start with verifying (up to a certain precision) that the abelian integral of $\omega_0$ vanishes between the Weierstrass points $R=(13,0), S=(1,0)$. Note that $R\in\pi^{-1}(U_3)$ and that $S\in \pi^{-1}(U_2)$. Consider the concatenation $\gamma=\gamma_1\gamma_2\gamma_3$ where $\gamma_1$ is a path from $S$ to $P_{e_1}$ in $\pi^{-1}(U_2)$, $\gamma_2$ is a path from $P_{e_1}$ to $P_{e_2}$ in $\pi^{-1}(U_1)$ and $\gamma_3$ is a path from $P_{e_2}$ to $R$ in $\pi^{-1}(U_3)$; hence $\tau(\gamma)=e_1e_2$. Since the tropical integral of $\eta$ along $e_1e_2$ is $0$, we have the equality
\[\Abint_S^R\omega_0 = \BCint_\gamma \omega_0 =  \BCint_S^{P_{e_1}} \omega_0+ \BCint_{P_{e_1}}^{P_{e_2}} \omega_0+\BCint_{P_{e_2}}^{R} \omega_0.\] 
Our methods yield
\begin{align*}
\BCint_S^{P_{e_1}} \omega_0 &= 2 \cdot a^{-1} + 8 \cdot a + 6 \cdot a^{3} + 9 \cdot a^{5} + 8 \cdot a^{7} + 3 \cdot a^{9} + 5 \cdot a^{11} + O(a^{13}),\\
\BCint_{P_{e_1}}^{P_{e_2}} \omega_0 &= 4 \cdot a^{-1} + 6 \cdot a + 3 \cdot a^{3} + 10 \cdot a^{5} + 8 \cdot a^{7} + 9 \cdot a^{9} + 11 \cdot a^{11} +  O(a^{13}),\\
\BCint_{P_{e_2}}^{R} \omega_0 &= 7 \cdot a^{-1} + 12 \cdot a + 3 \cdot a^{3} + 5 \cdot a^{5} + 9 \cdot a^{7} + 12 \cdot a^{9} + 8 \cdot a^{11} + O(a^{13}),
\end{align*} from which we get 
\[\Abint_S^R\omega_0 = O(a^{13})\] 
as required.

To demonstrate our methods, we compute the abelian integral of $\omega = \omega_1 + \omega_2$ between the following two  points lying in different basic wide opens:
\begin{align*}
R &= (13^3,2 \cdot 13^{3} + 2 \cdot 13^{4} + 10 \cdot 13^{5} + 11 \cdot 13^{6} + O(13^{7}))\in\pi^{-1}(U_4), \\
S &= (7,4 + 7 \cdot 13^{2} + 12 \cdot 13^{4} + 6 \cdot 13^{5} + O(13^{7}))\in\pi^{-1}(U_1) .
\end{align*}
Set $\gamma=\gamma_1\gamma_2\gamma_3$ where $\gamma_1$ is a path from $S$ to $P_{e_2}$ in $\pi^{-1}(U_1)$, $\gamma_2$ is a path from $P_{e_2}$ to $P_{e_3}$ in $\pi^{-1}(U_3)$ and $\gamma_3$ is a path from $P_{e_3}$ to $R$ in $\pi^{-1}(U_4)$; thus $\tau(\gamma)=e_2e_3$. For this path, we have
\begin{align*}
\BCint_\gamma \omega = &\ 11 + 4 \cdot a^{2} + 2 \cdot a^{4} + 10 \cdot a^{6} + 6 \cdot a^{8} + 7 \cdot a^{10} + 8 \cdot a^{12} + 9 \cdot a^{14}\\
&+ 9 \cdot a^{16} + 11 \cdot a^{18} + 9 \cdot a^{20} + 6 \cdot a^{22} + 4 \cdot a^{24} + 10 \cdot a^{26} + O(a^{28}),
\end{align*}
\[\tint_{\tau(\gamma)}\eta = \frac{1}{2}.\] 
Consider the loop $\gamma_C=\gamma_1\gamma_2\gamma_3\gamma_4$ in $X^{\an}$ where $\gamma_1$ is a path from $P_{v_3}$ to $P_{e_3}$ in $\pi^{-1}(U_3)$, $\gamma_2$ is a path from $P_{e_3}$ to $P_{v_4}$ in $\pi^{-1}(U_4)$, $\gamma_3$ is a path from $P_{v_4}$ to $P_{e_4}$ in $\pi^{-1}(U_4)$ and $\gamma_4$ is a path from $P_{e_4}$ to $P_{v_3}$ in $\pi^{-1}(U_3)$. The homology class of this loop is $C$ and the period integral is
\[\BCint_{\gamma_C} \omega = 8 \cdot a^{2} + 7 \cdot a^{6} + 2 \cdot a^{10} + 6 \cdot a^{14} + 10 \cdot a^{18} + 8 \cdot a^{26} + O(a^{28}).\] 
Consequently, we have 
\begin{eqnarray*}
\Abint_S^R \omega &=& \BCint_\gamma \omega - \left(\BCint_{\gamma_C}\omega\right) \left(\tint_{\tau(\gamma)} \eta\right)\\ &=& 11 + 2 \cdot a^4 + 6 \cdot a^8 + 8 \cdot a^{12} + 9 \cdot a^{16} + 9 \cdot a^{20} + 4 \cdot a^{24} + O(a^{28})\\ &=& 11 + 2 \cdot 13 + 6 \cdot 13^2 + 8 \cdot 13^3 + 9 \cdot 13^4 + 9 \cdot 13^5 + 4 \cdot 13^6 + O(13^7).
\end{eqnarray*}
\end{eg}

\begin{eg}(Chabauty--Coleman method)
Consider the even degree hyperelliptic curve $X/\Q$ \cite[\href{https://www.lmfdb.org/Genus2Curve/Q/3200/f/819200/1}{3200.f.819200.1}]{lmfdb} defined by the equation
\[y^2 = f(x) = (x^2-2)(x^2-x-1)(x^2+x-1).\] 
According to the database, this curve has exactly six rational points. In this final example, we will identify the annihilating differential to be used in the Chabauty--Coleman method at a prime of bad reduction. See the survey \cite{mccallum2012method} (especially Appendix A) for a detailed account of the method with many references.

The curve $X$ has bad reduction at the prime $5$ and its minimal regular model $\mathscr X$ over $\Z_5$ is given by the same equation as the above Weierstrass model. The Chabauty--Coleman bound \cite[Corollary~1.11]{lorenzin_tucker02:chabauty_coleman} (see also \cite[Theorem~1.4]{katz_dzb13:chabauty_coleman} for a refinement) gives
\[\# X(\Q)\leq \#\mathscr X_{\F_5}^{\operatorname{sm}}(\F_5)+2=8\]
where $\mathscr X_{\F_5}^{\operatorname{sm}}$ denotes the smooth locus of the special fiber of $\mathscr X$. A point count in Magma \cite{MR1484478} reveals the set of all rational points of naive height bounded by $10^5$:
\begin{eqnarray}
\label{Chabautyexample}
\{\infty^+,\infty^-,(1,\pm 1),(-1,\pm 1)\}\subseteq X(\Q).
\end{eqnarray}
Another computation in Magma \cite{MR1484478} shows that the Mordell-Weil rank of the Jacobian of $X$ is equal to $1$. Therefore, in order to check whether or not the curve $X$ has more rational points, one can use \cite[Theorem A.5.(1)]{mccallum2012method}. The crucial step is to construct the unique, up to a scalar multiple, annihilating differential on $X$. The fact that the known rational points are all in different residue discs makes it necessary to compute non-tiny integrals; this can be achieved by using of our techniques. 

The set $\{U_1,U_2,U_3\}$ is a good semistable covering of $\P^{1,\an}$ with respect to the set $S_f=\{\pm\sqrt{2}, \frac{1}{2}(1\pm\sqrt{5}), \frac{1}{2}(-1\pm\sqrt{5})\}$, where 
\begin{align*}
U_1 &= \P^{1,\an}\setminus\big(\overline{B}(1/2,1/\sqrt{5})\cup \overline{B}(-1/2,1/\sqrt{5})\big), \\
U_2 &= B(1/2,1), \\
U_3 &= B(-1/2,1),
\end{align*} with dual graph $T$
\[\begin{tikzpicture}
\draw [thick] (0,0) -- (6,0);
\filldraw (0,0) circle (2.5pt);
\filldraw (0,0.4) node{$U_2$};
\filldraw (3,0) circle (2.5pt);
\filldraw (3,0.4) node{$U_1$};
\filldraw (6,0) circle (2.5pt);
\filldraw (6,0.4) node{$U_3$};
\draw (0,0) -- (-0.7,-0.7);
\filldraw (-0.9,-0.9) node{\tiny{$\frac{1}{2}(1+\sqrt{5})$}};
\draw (0,0) -- (0.7,-0.7);
\filldraw (0.9,-0.9) node{\tiny{$\frac{1}{2}(1-\sqrt{5})$}};
\draw (3,0) -- (2.5,-0.7);
\filldraw (2.4,-0.9) node{\tiny{$\sqrt{2}$}};
\draw (3,0) -- (3.5,-0.7);
\filldraw (3.4,-0.9) node{\tiny{$-\sqrt{2}$}};
\draw (6,0) -- (5.3,-0.7);
\filldraw (5.1,-0.9) node{\tiny{$\frac{1}{2}(-1+\sqrt{5})$}};
\draw (6,0) -- (6.7,-0.7);
\filldraw (6.9,-0.9) node{\tiny{$\frac{1}{2}(-1-\sqrt{5})$}};
\end{tikzpicture}\]
Consider the points $R=(1,-1)$ and $S=(1,1)$, both belong to the space $\pi^{-1}(U_1)$. Therefore, for every holomorphic $1$-form $\omega$ on $X$, we have the equality
\[\Abint_S^R \omega = \BCint_S^R \omega.\]

The basic wide open $\pi^{-1}(U_1)$ is embedded into a rational curve:
\[\pi^{-1}(U_1)\simeq \{(x,\tilde{y})\mid\tilde{y}^2=x^2-2, x\in U_1\}\] 
where 
\[\tilde{y}=\frac{y}{\ell(x)},\ \ell(x) =\Big(x-\frac{1}{2}\Big)\Big(x+\frac{1}{2}\Big)\bigg(\Big(1-\frac{5/4}{(x-1/2)^2}\Big)\Big(1-\frac{5/4}{(x+1/2)^2}\Big)\bigg)^{1/2}.\]
Under this isomorphism,  the $1$-form ${\omega_i}_{\vert\pi^{-1}(U_1)}$ is represented by 
\[\bigg(\Big(1-\frac{5/4}{(x-1/2)^2}\Big)\Big(1-\frac{5/4}{(x+1/2)^2}\Big)\bigg)^{-1/2}\frac{x^i}{x^2-1/4}\frac{dx}{2\tilde{y}}.\]
Our methods yield
\begin{align*}
a\coloneq\BCint_S^R \omega_0 &= 2 \cdot 5 + 5^{4} + 3 \cdot 5^{6} + 2 \cdot 5^{7} + 2 \cdot 5^{8} + 4 \cdot 5^{9} + O(5^{10}),\\
b\coloneq\BCint_S^R \omega_1 &= O(5^{10}),
\end{align*} which give the annihilating differential as
\[\omega \coloneq b\omega_0-a\omega_1.\] 
It can be shown that the inclusion in \eqref{Chabautyexample} is actually an equality using the annihilating differential $\omega$. In particular, we have
\[X(\Q)=\{\infty^+,\infty^-,(1,\pm 1),(-1,\pm 1)\}.\]
\end{eg}

\begin{rem} We end this example with a remark about the importance of Chabauty--Coleman method at a prime of bad reduction. An illustration \cite[Example 5.1]{katz_dzb13:chabauty_coleman} was provided by the first-named author with Zureick-Brown. In this example, for a certain curve $X/\Q$ that has bad reduction at $5$, it is shown that $5$ is the only prime at which the refined Chabauty--Coleman bound \cite[Theorem 1.4]{katz_dzb13:chabauty_coleman} is sharp. Hence, one cannot determine the set $X(\Q)$ by using the Chabauty--Coleman bound at primes of good reduction alone; it is necessary to work with a prime of bad reduction or to make use of other techniques such as the Mordell--Weil sieve \cite{bruin_stoll_2010}.
\end{rem}

\bibliographystyle{amsalpha}
\bibliography{master}

\vspace{.2in}

\end{document}